\documentclass[11pt]{amsart}

\usepackage{amsmath}
\usepackage{amsfonts}
\usepackage{amssymb}
\usepackage{amscd}
\usepackage{amsthm}
\usepackage{framed}
\usepackage{fullpage}
\usepackage{graphicx}
\usepackage{latexsym}
\usepackage[numbers]{natbib}

\usepackage{hyperref}
\hypersetup{colorlinks,linkcolor={red},citecolor={blue},urlcolor={blue}}

\newtheorem{theorem}{Theorem}[section]
\newtheorem{lemma}[theorem]{Lemma}
\newtheorem{proposition}[theorem]{Proposition}
\newtheorem{corollary}[theorem]{Corollary}

\theoremstyle{definition}
\newtheorem{definition}[theorem]{Definition}
\newtheorem{example}[theorem]{Example}
\newtheorem{remark}[theorem]{Remark}

\newtheoremstyle{TheoremRef}
        {8.0pt plus 2.0pt minus 4.0pt}
        {8.0pt plus 2.0pt minus 4.0pt}  
        {\itshape}                      
        {}                              
        {\bfseries}                     
        {.}                             
        {5pt}                             
        {\thmname{#1}\thmnote{ \bfseries #3}}
\theoremstyle{TheoremRef}
\newtheorem{theorem_ref}{Theorem}[section]

\numberwithin{equation}{section}

\newcommand{\CC}{\mathbb C}
\newcommand{\HH}{\mathbb H}

\newcommand{\NN}{\mathbb N}
\newcommand{\cD}{\mathcal D}
\newcommand{\cA}{\mathcal A}

\newcommand{\PP}{\mathbb P}
\newcommand{\QQ}{\mathbb Q}
\newcommand{\RR}{\mathbb R}
\newcommand{\ZZ}{\mathbb Z}

\newcommand{\Mp}{\mathop{\mathrm {Mp}}\nolimits}
\newcommand{\SL}{\mathop{\mathrm {SL}}\nolimits}
\newcommand{\SO}{\mathop{\mathrm {SO}}\nolimits}

\newcommand{\Orth}{\mathop{\null\mathrm {O}}\nolimits}

\newcommand{\im}{\mathop{\mathrm {Im}}\nolimits}
\newcommand{\rank}{\mathop{\mathrm {rk}}\nolimits}
\newcommand{\latt}[1]{{\langle{#1}\rangle}}

\newcommand{\II}{\mathop{\mathrm {II}}\nolimits}
\def\Grit{\mathbf{G}}
\def\Borch{\mathbf{B}}
\def\m{\operatorname{mod}}

\def\det{\operatorname{det}}

\newenvironment{psmallmatrix}
  {\left(\begin{smallmatrix}}
{\end{smallmatrix}\right)}

\begin{document}

\title[Fake monster algebra and singular Borcherds products]{The fake monster algebra and singular Borcherds products}

\author{Haowu Wang}

\address{Center for Geometry and Physics, Institute for Basic Science (IBS), Pohang 37673, Korea}

\email{haowu.wangmath@gmail.com}

\author{Brandon Williams}

\address{Lehrstuhl A für Mathematik, RWTH Aachen, 52056 Aachen, Germany}

\email{brandon.williams@matha.rwth-aachen.de}


\date{\today}


\begin{abstract} 
In this paper we consider several problems in the theory of automorphic products and generalized Kac--Moody algebras proposed by Borcherds in 1995. We show that the denominator of the fake monster algebra defines the unique holomorphic Borcherds product of singular weight on a maximal lattice. We give a full classification of symmetric holomorphic Borcherds products of singular weight on lattices of prime level. Finally we prove that all twisted denominator identities of the fake monster algebra arise as the Fourier expansions of Borcherds products of singular weight at a certain cusp. The proofs rely on an identification between modular forms for the Weil representation attached to lattices of type $U(N)\oplus U \oplus L$ and certain tuples of Jacobi forms of level $N$.
\end{abstract}

\maketitle

\begin{small}
\tableofcontents
\end{small}

\section{Introduction}
Borcherds defined generalized Kac--Moody algebras \cite{Bor88} in 1988 and applied this concept to his celebrated proof of the monstrous moonshine conjecture \cite{Bor92} in 1992. Later, he observed that the denominator functions of some generalized Kac--Moody algebras are modular forms on orthogonal groups $\Orth(l,2)$. In 1995 and 1998 Borcherds developed the theory of automorphic products to construct orthogonal modular forms as infinite products \cite{Bor95, Bor98}. Let $M$ be an even lattice of signature $(l,2)$ with $l\geq 1$ and $\rho_M$ be the Weil representation of $\SL_2(\ZZ)$ attached to the discriminant form of $M$. Given a weakly holomorphic modular form $F$ of weight $1-l/2$ for $\rho_M$ with integral Fourier expansion. By Borcherds' theory, the singular theta lift of $F$ gives a meromorphic modular form $\Borch(F)$ of weight $c_0(0)/2$ on a certain subgroup of $\Orth(M)$, which has nice product expansions at any cusps and whose divisor is a linear combination of rational quadratic divisors, where $c_0(0)$ is the constant coefficient in the Fourier expansion of $F$. Such $\Borch(F)$ is called a Borcherds product on $M$. 

When $l\geq 3$, the weight of a non-constant holomorphic modular form on $\Orth(l,2)$ is at least $l/2-1$. The smallest possible positive weight $l/2-1$ is called the \textit{singular} weight. Modular forms of singular weight are particularly interesting because their Fourier coefficients are supported only on isotropic vectors. Holomorphic Borcherds products of singular weight are very exceptional objects. It is expected that generalized Kac--Moody algebras whose denominator identities are singular Borcherds products have natural constructions (i.e. not simply listing generators and relations) and interesting symmetry groups (\cite{Bor90, Sch00, HS03, CKS07, HS14, Car16, Mol21}). For example, the physical states of a bosonic string moving on a $26$-dimensional torus define the fake monster algebra denoted $G$ \cite{Bor90}. The fake monster algebra has root lattice $\II_{25,1}$ and its denominator identity is
\begin{equation}\label{eq:intro-denom}
e^{\rho} \prod_{\alpha \in \II_{25,1}^+} \Big(1-e^\alpha\Big)^{[\frac{1}{\Delta}](-\alpha^2/2)} = \sum_{w \in W} \det(w) w\left( e^{\rho}\prod_{n=1}^\infty\Big( 1 - e^{n\rho}\Big)^{24} \right),    
\end{equation}
which defines a holomorphic Borcherds product of singular weight $12$ on $\Orth^+(\II_{26,2})$ which vanishes precisely on rational quadratic divisors orthogonal to 2-roots of $\II_{26,2}$. This is the so-called Borcherds' form denoted by $\Phi_{12}$. 

Equation \eqref{eq:intro-denom} comes from a cohomological identity. The orthogonal group of the Leech lattice $\Lambda$ is the Conway group $\mathrm{Co}_0$, and a certain extension of $\mathrm{Co}_0$ acts naturally on the fake monster algebra $G$. By considering the action of any $g \in \mathrm{Co}_0$ of cycle shape $\prod k^{b_k}$ on $G$ and taking traces, Borcherds \cite{Bor92} established the twisted denominator identities of the fake monster algebra,
\begin{equation}
e^\rho \prod_{\alpha \in L_g^+}\Big(1-e^\alpha\Big)^{\mathrm{mult}(\alpha)} = \sum_{w\in W_g} \det(w) w\left( e^\rho \prod_{k=1}^\infty \prod_{n=1}^\infty \Big(1-e^{kn\rho}\Big)^{b_k} \right)    
\end{equation}
and proved that each of them is the usual denominator identity of a generalized Kac--Moody superalgebra, where $L_g=U\oplus \Lambda^g$, $U$ is the unique even unimodular lattice of signature $(1,1)$, $\Lambda^g$ is the fixed-point sublattice of $\Lambda$ associated with $g$, and $\mathrm{mult}(\alpha)$ are defined in terms of traces of powers of $g$ acting on $G$.

In 1995 Borcherds proposed the following problems:
\begin{enumerate}
    \item Classify Borcherds products of singular weight and the associated generalized Kac--Moody algebras (\cite[\S 17, Problem 3]{Bor95} and \cite[Problem 16.2]{Bor98});
    \item Prove that the twisted denominator identities define automorphic forms of singular weight for some orthogonal groups (\cite[\S 15, Example 3]{Bor95}).
\end{enumerate}

Holomorphic Borcherds products of singular weight are closely related to reflective modular forms introduced by Borcherds \cite{Bor98} and Gritsenko--Nikulin \cite{GN98}. A non-constant modular form on $M$ is called \textit{reflective} if its zeros lie on rational quadratic divisors orthogonal to roots of $M$. Reflective modular forms of arbitrary weight have important applications to hyperbolic reflection groups \cite{Bor00}, moduli spaces \cite{GHS07, Ma17, Gri18}, infinite dimensional Lie algebras \cite{GN98, GN02, Sch06, GN18} and the classification and construction of free algebras of modular forms \cite{Wan21a}. Recently, we \cite{WW23} proved that every singular Borcherds product on a lattice $M$ arises as a reflective modular form on a certain even lattice contained in $M\otimes\QQ$ (may be different from $M$). This result implies that a complete classification of reflective modular forms leads to a satisfactory solution to problem (1).  Over the past two decades, many classifications of reflective modular forms have been obtained \cite{GN02, Bar03, Sch06, Sch17, Ma17, Ma18, Dit19, Wan21, Wan19, Wan22, Wan23a, Wan23b}, but no full classification is yet available. 

In this paper we prove several new results related to problem (1) and present a complete solution to problem (2). 

We classify singular Borcherds products in certain families of lattices without a priori that these products are reflective.   There are only a few known results in this direction. In \cite{DHS15} Dittmann, Hagemeier and Schwagenscheidt classified simple lattices (i.e. the lattices on which there is no obstruction to construct Borcherds products) of square-free level and the singular Borcherds products on them. Later, Opitz and Schwagenscheidt \cite{OS19} classified singular Borcherds products on simple lattices of arbitrary level. There are only finitely many simple lattices and a full classification was given in \cite{BEF16}. Note that \cite{DHS15, OS19} only classified singular Borcherds products coming from weakly holomorphic modular forms whose principal parts are non-negative; there are known examples of singular Borcherds products for which this does not hold (cf. \cite{GN98}). The first classification of singular Borcherds products on infinite families of lattices was achieved by Scheithauer \cite{Sch17}. Scheithauer proved that the Borcherds form $\Phi_{12}$ is the unique holomorphic Borcherds product of singular weight on unimodular lattices. He also showed that singular Borcherds products on lattices of prime level exist only in small signatures and gave an explicit bound. However, Scheithauer's bound depends on the discriminant form, so it does not lead to a full classification. 

In this paper we improve on Scheithauer's results. We first classify all singular Borcherds products on maximal lattices. If $\Borch(F)$ is a singular Borcherds product on a maximal lattice $M$ then we will show that $\Borch(F)$ is non-vanishing at a 1-dimensional cusp related to a decomposition $M=2U\oplus L$. Gritsenko and Nikulin determined the Fourier--Jacobi expansions of Borcherds products on lattices of type $2U\oplus L$ (see \cite{GN98, Gri18}). Their result implies that the zeroth Fourier--Jacobi coefficient of $\Borch(F)$ has to be $\Delta=\eta^{24}$. We use this to prove:

\begin{theorem_ref}[\ref{th:maximal}]
The Borcherds form $\Phi_{12}$ is the unique holomorphic Borcherds product of singular weight on a maximal lattice. 
\end{theorem_ref}

We extend the argument in the proof of Theorem \ref{th:maximal} to classify symmetric Borcherds products of singular weight on lattices of prime level. Following \cite{Sch06} a Borcherds product on $M$ is called \textit{symmetric} if it is modular for the full orthogonal group $\Orth^+(M)$. This condition is necessary to exclude some infinite families of pullbacks of singular Borcherds products.  

\begin{theorem_ref}[\ref{th:prime-level}]
Besides the infinite family of rescaled lattices $\II_{26,2}(p)$ for primes $p$, there are exactly $12$ lattices of prime level which admit a symmetric holomorphic Borcherds product of singular weight. For each such lattice, the singular Borcherds product is unique. The genera of the $12$ lattices are as follows
\begin{align*}
    &\II_{18,2}(2_{\II}^{+10})& &\II_{10,2}(2_{\II}^{+2})& &\II_{10,2}(2_{\II}^{+10})& &\II_{14,2}(3^{-8})& &\II_{8,2}(3^{-3})& &\II_{8,2}(3^{-7})&\\
    &\II_{10,2}(5^{+6})& &\II_{6,2}(5^{+3})& &\II_{6,2}(5^{+5})& &\II_{8,2}(7^{-5})& &\II_{6,2}(11^{-4})& &\II_{4,2}(23^{-3}).&
\end{align*}
\end{theorem_ref}

Note that $\Phi_{12}$ defines the unique symmetric Borcherds product of singular weight on $\II_{26,2}(p)$ for any prime $p$. 
Scheithauer \cite{Sch06}  proved that the above $12$ lattices are exactly the lattices of prime level which have a symmetric reflective Borcherds product of singular weight whose input is also reflective. The $12$ singular products can be identified with the twisted denominator functions of the fake monster algebra corresponding to elements of $\mathrm{Co}_0$ of prime level and non-trivial fixed-point sublattice. 

The proof of Theorem \ref{th:prime-level} relies on a description of the Fourier--Jacobi expansion of a Borcherds product on lattices of type $U(N)\oplus U \oplus L$ (see Theorem \ref{th:FJ-level-N}) and an identification of the input in Borcherds' lift with certain sequences of weakly holomorphic Jacobi forms (see Theorem \ref{th:iso} below). Let $F$ be a symmetric singular Borcherds product on a lattice $M$ of prime level $p$. Assume that $M$ is not of type $\II_{26,2}(p)$. We will show that $F$ does not vanish at a $1$-dimensional cusp related to a decomposition $M=U(p)\oplus U\oplus L$ and the corresponding zeroth Fourier--Jacobi coefficient is an eta quotient. We then prove Theorem \ref{th:prime-level} by analyzing the input of $F$ as a pair of Jacobi forms. 

\begin{theorem_ref}[\ref{th:iso}]
Let $M=U(N)\oplus U\oplus L$.
There is an isomorphism between the space of weakly holomorphic modular forms of weight $k-\frac{1}{2}\rank(L)$ for $\rho_M$ and the space of sequences of Jacobi forms
$$
\bigoplus_{d|N} J_{k,L}^!(\Gamma_1(N/d)),
$$
where $J_{k,L}^!(\Gamma_1(t))$ is the space of weakly holomorphic Jacobi forms of weight $k$ and index $L$ on $\Gamma_1(t)$. 
\end{theorem_ref}

This is a generalization of the classical isomorphism between modular forms for $\rho_L$ and Jacobi forms of index $L$ on $\SL_2(\ZZ)$ (see e.g. \cite{Gri94}). The construction is motivated by \cite{Bor95, Car12} in which Borcherds and Carnahan realized modular forms for $\rho_{U(N)}$ in terms of scalar-valued modular forms on $\Gamma_1(t)$ for $t|N$. 

We use the above isomorphism to resolve Borcherds' problem (2). Scheithauer made significant contributions to this problem. First, in \cite{Sch04, Sch06} he formulated Borcherds' problem precisely as the moonshine conjecture for Conway's group. This claims that the twisted denominator identity of the fake monster algebra corresponding to $g\in \mathrm{Co}_0$ defines a Borcherds product of weight $\frac{1}{2}\rank(\Lambda^g)$ on some orthogonal group of signature $(2+\rank(\Lambda^g),2)$. Then he proved this conjecture for all Conway's elements of square-free level and eight additional elements \cite{Sch01, Sch04, Sch08, Sch09, Sch15, Sch17}. To prove these cases Scheithauer lifted the inverse of the eta quotient associated with the cycle shape of $g$ to a modular form for the Weil representation and showed that the associated Borcherds product has the correct Fourier expansion.

Carnahan's work \cite{Car12} is also related to this problem. In his famous proof of Conway and Norton’s moonshine conjecture \cite{Bor92}, Borcherds constructed the twisted denominator identity of the monster algebra by considering the action of the monster group, i.e. the largest sporadic simple group. Borcherds also proved that this identity defines an automorphic form of weight $0$ on some orthogonal group of signature $(2,2)$. Carnahan further showed that the twisted denominator identities can be realized as Borcherds products. We know from \cite{CN79} that for each element of $\mathrm{Co}_0$ with trivial fixed-point sublattice the inverse of the associated eta quotient equals the McKay--Thompson series of some element in the monster up to additive constant. Therefore, Carnahan's result implies that the moonshine conjecture holds for all elements of $\mathrm{Co}_0$ with trivial fixed-point sublattice.  

In this paper we prove the moonshine conjecture for all elements of Conway's group using the theory of Jacobi forms:

\begin{theorem_ref}[\ref{th:moonshine}]
Let $g$ be an element of $\mathrm{Co}_0$ of level $N_g$ with fixed-point sublattice $\Lambda^g$.  Then the twisted denominator identity of the fake monster algebra corresponding to $g$ defines a Borcherds product of weight $\frac{1}{2}\rank(\Lambda^g)$ on the lattice $U(N_g)\oplus U\oplus \Lambda^g$. 
\end{theorem_ref}

We prove Theorem \ref{th:moonshine} by constructing, for each $d|N_g$, a weakly holomorphic Jacobi form $\phi_{g,d}$ of weight $0$ and index $\Lambda^g$ on $\Gamma_0(N_g/d)$ in terms of the traces of a certain lift of $g^d$ acting on subspaces of the vertex operator algebra of the Leech lattice. We then show that the Borcherds lift of the image of $(\phi_{g,d})_{d|N_g}$ under the isomorphism of Theorem \ref{th:iso} is exactly the twisted denominator function of the fake monster algebra corresponding to $g$. The proof does not require any explicit computation of the principal part of the image of $(\phi_{g,d})_{d|N_g}$. 

The proof fits into an intriguing relation inspired by Borcherds' works on vertex algebras, generalized Kac--Moody algebras and automorphic forms (see e.g. diagram at the end of \cite[\S 4]{Mol21}). As mentioned before, a generalized Kac--Moody algebra $\mathfrak{g}$ whose denominator $\Phi_\mathfrak{g}$ is a singular Borcherds product is conjectured to have a natural construction, given by the BRST cohomology of some suitable vertex algebra $V_{\mathfrak{g}}$. Moreover, one conjectures that the input forms that lift to $\Phi_\mathfrak{g}$ are vector-valued characters of $V_{\mathfrak{g}}$. For the generalized Kac--Moody superalgebras obtained by twisting the fake monster algebra, Theorem \ref{th:moonshine} implies that the inputs of the twisted denominators are essentially vector-valued characters of certain vertex algebras related to the Leech vertex operator algebra (see Remark \ref{rk:VOA}), so it is very likely that these algebras have a unified string-theoretic construction via the BRST cohomology. This construction was recently realized by M\"{o}ller \cite{Mol21} for the ten superalgebras corresponding to elements of square-free level in the Mathieu group $M_{23}<\mathrm{Co}_0$.

The paper is organized as follows. In \S \ref{sec:pre} we review some basics of modular forms on orthogonal groups and Jacobi forms of lattice index. \S \ref{sec:maximal} contains the proof of Theorem \ref{th:maximal}. In \S \ref{sec:isomorphism} we prove Theorem \ref{th:iso} and study the Fourier--Jacobi expansion of Borcherds products on $U(N)\oplus U\oplus L$. We prove Theorem \ref{th:prime-level} in \S \ref{sec:prime}. Finally, we prove Theorem \ref{th:moonshine} in \S \ref{sec:moonshine}.

\section{Preliminaries}\label{sec:pre}
In this section we review some basic properties of orthogonal modular forms and Jacobi forms.

\subsection{Modular forms on orthogonal groups}
Let $M$ be an even integral lattice of signature $(l,2)$ with $l\geq 3$. The complex manifold
$$
\{\mathcal{Z} \in  M\otimes \CC:  (\mathcal{Z}, \mathcal{Z})=0, (\mathcal{Z},\bar{\mathcal{Z}}) < 0\}
$$
has two connected components which are exchanged by the conjugate $\mathcal{Z} \mapsto \bar{\mathcal{Z}}$. We choose one of them and label it $\cA(M)$.  The quotient of $\cA(M)$ by $\mathbb{C}^{\times}$ is the symmetric domain of type IV attached to $M$:
$$
\cD(M)=\{[\mathcal{Z}] \in  \PP(M\otimes \CC):  \mathcal{Z} \in \cA(M) \}.
$$
Let $\Orth^+(M)$ denote the subgroup of $\Orth(M\otimes\RR)$ which preserves $M$ and $\cA(M)$.  Let $\Gamma$ be a finite-index subgroup of $\Orth^+(M)$. The most important example of $\Gamma$ is the \textit{discriminant kernel}, defined as
$$
\widetilde{\Orth}^+(M)=\{ g \in \Orth^+(M):\; g(v) - v \in M, \; \text{for all $v\in M'$} \},
$$
where $M'$ is the dual lattice of $M$. 

\begin{definition}\label{def: mf}
Let $k\in \ZZ$ and $\chi: \Gamma\to \CC^\times$ be a character. A holomorphic function $F: \cA(M)\to \CC$ is called a modular form of weight $k$ and character $\chi$ on $\Gamma$ if it satisfies
\begin{align*}
F(t\mathcal{Z})&=t^{-k}F(\mathcal{Z}), \quad  \text{for all $t \in \CC^\times$},\\
F(g\mathcal{Z})&=\chi(g)F(\mathcal{Z}), \quad \text{ for all $g\in \Gamma$}.
\end{align*}
\end{definition}
If $F$ is non-zero, then either $k=0$ in which case $F$ is constant, or $k\geq l/2-1$. The smallest possible positive weight $l/2-1$ is called the \textit{singular weight}.

Modular forms can be expanded into Fourier series on the tube domain around any $0$-dimensional cusp. Let $c$ be a primitive isotropic vector of $M$ and choose $c'\in M'$ satisfying $(c,c')=1$. Then $M_{c,c'}=M\cap c^\perp \cap (c')^\perp$ is an even lattice of signature $(l-1,1)$. The tube domain $\HH_{c,c'}$ is the connected component of
$$
\{ Z = X + iY :\; X, Y \in M_{c,c'}\otimes\RR, \; (Y,Y)<0  \}, 
$$
which embeds into $\cA(M)$ via the map
$$
\phi_{c,c'}: \HH_{c,c'} \to \cA(M), \quad Z \mapsto c'+Z-\frac{(Z,Z)+(c',c')}{2}c.
$$
Suppose $F$ is a modular form of weight $k$ and trivial character on $\widetilde{\SO}^+(M)$. Then $F$ is represented on $\HH_{c,c'}$ by the Fourier series
$$
F(Z)=\sum_{\substack{\lambda \in M_{c,c'}'\\(\lambda, \lambda)\leq 0}} c(\lambda) e^{2\pi i(\lambda, Z)}.
$$
Modular forms $F$ of general level have similar expansions with $M_{c,c'}$ replaced by a finite-index sublattice. When $F$ has singular weight, its Fourier coefficients $c(\lambda)$ are zero whenever $(\lambda,\lambda)\neq 0$.

\subsection{Modular forms for the Weil representation}
An even lattice $M$ of signature $(b^+,b^-)$ induces a discriminant form $D_M:=(M'/M, Q)$ with the quadratic form
$$
Q : M'/M \rightarrow \mathbb{Q}/\mathbb{Z}, \; Q(x + M) = (x,x)/2 + \mathbb{Z}.
$$
Let $\mathrm{Mp}_2(\mathbb{Z})$ be the metaplectic group, which consists of pairs $A = (A, \phi_A)$ where $A = \begin{psmallmatrix} a & b \\ c & d \end{psmallmatrix} \in \mathrm{SL}_2(\mathbb{Z})$ and $\phi_A$ is a holomorphic square root of $\tau \mapsto c \tau + d$ on $\mathbb{H}$, with the standard generators $T = (\begin{psmallmatrix} 1 & 1 \\ 0 & 1 \end{psmallmatrix}, 1)$ and $S = (\begin{psmallmatrix} 0 & -1 \\ 1 & 0 \end{psmallmatrix}, \sqrt{\tau})$. The \emph{Weil representation} $\rho_M$ is the representation of $\mathrm{Mp}_2(\mathbb{Z})$ on the group ring $\mathbb{C}[D_M] = \mathrm{span}(e_x: \; x \in D_M)$ defined by 
$$
\rho_M(T) e_x = \mathbf{e}(-Q(x)) e_x \quad \text{and} \quad \rho_M(S) e_x = \frac{\mathbf{e}( \mathrm{sign}(M) / 8)}{\sqrt{|D_M|}} \sum_{y \in D_M} \mathbf{e}(( x,y)) e_y,
$$
where $\mathbf{e}(t)=e^{2\pi it}$ for $t\in \CC$, and $\mathrm{sign}(M)=b^+ - b^- \mod 8$. Let $\bar{\rho}_M$ denote the complex conjugate of $\rho_M$, which is the dual representation of $\rho_M$. Note that $\bar{\rho}_M=\rho_{M(-1)}$. 

A \emph{weakly holomorphic modular form} of weight $k \in \frac{1}{2}\mathbb{Z}$ for the Weil representation $\rho_M$ is a holomorphic function $f : \mathbb{H} \rightarrow \mathbb{C}[D_M]$ that satisfies 
$$
f\Big|_k A(\tau)  := \phi_A(\tau)^{-2k} f(A \cdot \tau) = \rho_M(A) f(\tau), \quad \text{for all $A \in \mathrm{Mp}_2(\mathbb{Z})$,}
$$ 
and which is meromorphic at infinity; that is, $f$ is represented by a Fourier series of the form 
\begin{equation}\label{eq:Fourier}
f(\tau) = \sum_{x \in D_M} \sum_{\substack{n \in \mathbb{Z} - Q(x)\\ n \gg -\infty}} c_x(n) q^n e_x.    
\end{equation}
The finite sum of Fourier coefficients $c_x(n)q^n e_x$ with $n<0$ is called the \textit{principal part} of $f$. 
We further call $f$ a \textit{holomorphic modular form} if it is holomorphic at infinity, i.e. its principal part is zero. We remark that if $f$ is non-zero then $k+\mathrm{sign}(M)/2 \in \ZZ$, and if $\mathrm{sign}(M)$ is even then $\rho_M$ factors through a representation of $\SL_2(\ZZ)$. We denote the spaces of weakly holomorphic modular forms and holomorphic modular forms of weight $k$ for $\rho_M$ by $M_k^!(\rho_M)$ and $M_k(\rho_M)$, respectively. 

\begin{example}\label{ex:Theta}
Let $L$ be an even positive-definite lattice. One defines the Jacobi theta function associated with any $\gamma \in L'/L$ as
$$
\Theta_{L,\gamma}(\tau,\mathfrak{z}) = \sum_{\ell \in L+\gamma} e^{\pi i(\ell,\ell)\tau + 2\pi i(\ell, \mathfrak{z})}, \quad (\tau,\mathfrak{z})\in\HH\times (L\otimes\CC). 
$$
Then the function 
$$
\mathbf{\Theta}_L(\tau, \mathfrak{z}) = \sum_{x \in L'/L} \Theta_{L,x}(\tau,\mathfrak{z})e_x
$$ 
is a holomorphic Jacobi form of weight $\frac{1}{2}\rank(L)$ and index $L$ and multiplier $\bar{\rho}_L$, i.e. 
\begin{align*}
\mathbf{\Theta}_L \Big( \frac{a \tau + b}{c \tau + d}, \frac{\mathfrak{z}}{c \tau + d} \Big) &= (c \tau + d)^k e^{\pi i (\mathfrak{z},\mathfrak{z}) c / (c \tau + d)} \bar{\rho}_L(A) \mathbf{\Theta}_L(\tau, \mathfrak{z}), \quad A=\begin{pmatrix} a & b \\ c & d \end{pmatrix} \in \mathrm{Mp}_2(\mathbb{Z}),\\
\mathbf{\Theta}_L(\tau, \mathfrak{z} + \lambda \tau + \mu) &= e^{-\pi i ( (\lambda, \lambda)\tau + 2( \lambda, \mathfrak{z}))} \mathbf{\Theta}_L(\tau, \mathfrak{z}), \quad \lambda, \mu \in L.
\end{align*}
In particular, $\mathbf{\Theta}_L(\tau,0)$ is a modular form of weight $\frac{1}{2}\rank(L)$ for $\bar{\rho}_L$. 
\end{example}

Let $N$ denote the level of $M$, i.e. the smallest positive integer $N$ such that $N(v,v)\in 2\ZZ$ for all $v\in M'$. We define $\widetilde{\Gamma}_0(N)$ as the inverse image of $\Gamma_0(N)$ under the natural map $\Mp_2(\ZZ) \to \SL_2(\ZZ)$.
By \cite[Lemma 3.2]{Bor00}, for any $A\in \widetilde{\Gamma}_0(N)$ there exists $\chi_{D_M}(A) \in \CC^\times$ such that
$$
\rho_M(A) e_0 = \chi_{D_M}(A)e_0.
$$
Thus $\chi_{D_M}$ defines a character of $\widetilde{\Gamma}_0(N)$. In \cite{Bor00} Borcherds proposed the following lifting construction.

\begin{theorem}\label{th:scalar-to-vector}
Let $f$ be a weakly holomorphic scalar-valued modular form of weight $k\in \frac{1}{2}\ZZ$ and character $\chi_{D_M}$ for $\widetilde{\Gamma}_0(N)$. Then 
$$
F_{\Gamma_0(N),f,0}(\tau)=\sum_{A\in \widetilde{\Gamma}_0(N)\backslash \Mp_2(\ZZ)} f\Big|_k A(\tau) \rho_M(A^{-1})e_0
$$ 
defines a weakly holomorphic modular form of weight $k$ for $\rho_M$ which is invariant under $\Orth(D_M)$. 
\end{theorem}

\begin{remark}\label{rk:Scheithauer-lifting}
Scheithauer calculated this lifting explicitly when $\mathrm{sign}(M)$ is even (see \cite[Theorems 5.4, 5.7]{Sch09}). We recall his result in the special case that the level of $M$ is a prime $p$, because only this case will be needed later. Then 
$$
\chi_{D_M}(A)=\left( \frac{a}{|D_M|}\right), \quad A= \begin{pmatrix} a & b \\ c & d \end{pmatrix} \in \Gamma_0(p). 
$$
If we write 
$$ 
f\Big|_k S (\tau)=\sum_{j=0}^{p-1}g_j(\tau),  \quad S=\begin{pmatrix} 0 & -1 \\ 1 & 0 \end{pmatrix},
$$
where 
$$g_j(\tau+1)=\exp(2\pi ij/p) g_j(\tau), \quad 0\leq j 
\leq p-1,$$
then we have 
\begin{equation*}
F_{\Gamma_0(p),f,0}(\tau)=f(\tau)e_0+\xi_1\frac{p}{\sqrt{\lvert D_M
\rvert}}\sum_{\gamma\in D_M} g_{j_\gamma}(\tau) e_\gamma,
\end{equation*}
where $j_\gamma/p = -(\gamma,\gamma)/2 \mod 1$ for $\gamma\in D_M$ 
and 
$$\xi_1=\left(\frac{-1}{\lvert D_M\rvert}\right)\exp\left( \frac{\mathrm{sign}(M) \pi i}{4}\right).
$$
\end{remark}

\subsection{Borcherds products}\label{subsec:Borcherds-products} Let $M$ be an even lattice of signature $(l,2)$ with $l\geq 3$. Let $f$ be a weakly holomorphic modular form of weight $1-l/2$ for $\rho_M$ with integral principal part (see \eqref{eq:Fourier}). The Borcherds multiplicative lift \cite{Bor98} produces a meromorphic modular form $\Borch(f)$ of weight $c_0(0)/2$ and some character (or multiplier system) on $\widetilde{\Orth}^+(M)$ which satisfies
\begin{enumerate}
    \item All zeros or poles of $\Borch(f)$ lie on rational quadratic divisors $\lambda^\perp$, where $\lambda \in M'$ is a primitive vector of positive norm. The multiplicity of $\lambda^\perp$ in the divisor is given by
    $$
    \sum_{d=1}^\infty c_{d\lambda}(-d^2\lambda^2/2). 
    $$
    \item About any $0$-dimensional cusp $c$, $\Borch(f)$ has an infinite product expansion on the associated tube domain $\HH_{c,c'}$ in which the exponents are Fourier coefficients of $f$. 
\end{enumerate}
Suppose $M=U(N)\oplus K$ for some signature $(\ell-1, 1)$ lattice $K$. We will need the infinite product expansion of $\Borch(f)$ at the level $N$ zero-dimensional cusp related to this splitting. Let $e\in U(N)$ be a primitive vector of norm $0$ and fix $e' \in U(N)'=U(1/N)$ satisfying $(e',e')=0$ and $(e,e')=1$. Then $\Borch(f)$ is represented on an open subset of $\HH_{e,e'}$ by the product
\begin{equation}\label{eq:Bor-Fourier}
\Borch(f)(Z) = \mathbf{e}((\rho, Z)) \prod_{\substack{\lambda\in K'\\ \lambda>0}}\prod_{a\m N} \left( 1- \mathbf{e}\left(\frac{a}{N} + (\lambda,Z)\right) \right)^{c_{\lambda + \frac{a}{N}e}(-\lambda^2/2)},
\end{equation}
where $\rho$ is the Weyl vector. 

\subsection{Jacobi forms of lattice index}\label{subsec:Jacobi-forms}
 One defines the Jacobi form of lattice index (see \cite{Gri94}) as a generalization of classical Jacobi forms introduced in \cite{EZ85}. Let $L$ be an even positive-definite lattice. Let $\Gamma$ be a congruence subgroup of $\SL_2(\ZZ)$ and let $\chi: \Gamma \to \CC^\times$ be a character. 

\begin{definition}
Let $k\in\ZZ$. A holomorphic function $\varphi : \HH \times (L \otimes \CC) \rightarrow \CC$ is 
called a {\it weakly holomorphic} Jacobi form of weight $k$, character $\chi$ and index $L$ on $\Gamma$ if it satisfies 
$$
\varphi\Big|_{k,L}A(\tau,\mathfrak{z}):=(c\tau + d)^{-k} 
\exp{\left(-i \pi  \frac{c(\mathfrak{z},\mathfrak{z})}{c 
\tau + d}\right)} \varphi \left( \frac{a\tau +b}{c\tau + d},\frac{\mathfrak{z}}{c\tau + d} 
\right) = \chi(A)\varphi ( \tau, \mathfrak{z} )
$$
for $A=\begin{psmallmatrix} a & b \\ c & d \end{psmallmatrix} \in \Gamma$ and
$$
\varphi (\tau, \mathfrak{z}+ x \tau + y)= 
\exp{\bigl(-i \pi ( (x,x)\tau +2(x,\mathfrak{z}))\bigr)} 
\varphi (\tau, \mathfrak{z} ), \quad x,y \in L,
$$
and if its Fourier expansion at each cusp $\mathbf{c}$ (represented by $A_\mathbf{c}\in \SL_2(\ZZ)$) of $\Gamma$ takes the form
\begin{equation*}
\varphi\Big|_{k,L}A_\mathbf{c}(\tau,\mathfrak{z})= \sum_{n\gg -\infty}\sum_{\ell\in L'}f_{\mathbf{c}}(n,
\ell)q^n\zeta^\ell,
\end{equation*}
where $q=e^{2\pi i \tau}$ and $\zeta^\ell=e^{2\pi i (\ell,
\mathfrak{z})}$. 
If at every cusp $f_{\mathbf{c}}(n,\ell) = 0$ whenever $2n-(\ell,\ell)<0$,
then $\varphi$ is called a {\it holomorphic} Jacobi form.  We denote the space of weakly holomorphic and holomorphic Jacobi forms by $J_{k,L}^!(\Gamma, \chi)$ and $J_{k,L}(\Gamma, \chi)$, respectively. 
\end{definition}

Jacobi forms appear in the Fourier--Jacobi expansion of an orthogonal modular form at a $1$-dimensional cusp. 
Let $M=U(N)\oplus U_1\oplus L$ and $F$ be a modular form of weight $k$ and trivial character on $\widetilde{\SO}^+(M)$. As in the previous subsection, we consider the Fourier expansion of $F$ at the 0-dimensional cusp related to $U(N)$. Fix a basis of the second hyperbolic plane $U_1=\ZZ e_1+\ZZ f_1$ with $(e_1,e_1)=(f_1,f_1)=0$ and $(e_1,f_1)=-1$. The associated tube domain as 
$$
\HH_{e,e'} = \{ Z = -\tau f_1 + \mathfrak{z} - \omega e_1 :\; \tau, \omega \in \HH, \; \mathfrak{z} \in L\otimes\CC,\; 2\im(\tau)\im(\omega) - (\mathfrak{z},\mathfrak{z})>0 \}.
$$
This coordinate system is chosen such that $(\alpha,Z)=n\tau +(\ell,\mathfrak{z})+ m\omega$ for $\alpha=ne_1 + \ell + m f_1$ with $n,m\in \ZZ$ and $\ell \in L'$.
Then $F$ has the following Fourier--Jacobi expansion on $\HH_{e,e'}$:
$$
F(Z)=\sum_{m=0}^\infty \sum_{n=0}^\infty \sum_{\substack{\ell \in L' \\ 2nm\geq (\ell,\ell)}} f(n,\ell,m)q^n\zeta^\ell s^m = \sum_{m=0}^\infty \phi_m(\tau,\mathfrak{z})s^m, \quad s=e^{2\pi i\omega}.
$$
Let $H(L)$ be the integral Heisenberg group of $L$ (see e.g. \cite{Gri94}). It is easy to check that $\Gamma_0(N) \ltimes H(L)$ embeds into $\Orth^+(M)$ and preserves the isotropic plane spanned by $e$ and $e_1$. Moreover, $\Gamma_1(N) \ltimes H(L)$ embeds into $\widetilde{\Orth}^+(M)$. Therefore, for any $m$ the function $\phi_m$ defined above is a holomorphic Jacobi form of weight $k$ and index $L(m)$ on $\Gamma_1(N)$. If $F$ is a modular form on $\Orth^+(M)$, then $\phi_m$ is a Jacobi form on $\Gamma_0(N)$. 

Following \cite[\S 2]{CG11} we define index-raising Hecke operators for Jacobi forms on $\Gamma_0(N)$. Let $\phi \in J_{k,L}^!(\Gamma_0(N), \chi_N)$, where $\chi_N$ is a Dirichlet character of modulus $N$. For any positive integer $m$ we have
\begin{equation}\label{eq:index-raising}
\phi\Big|_{k,L} T_{-}^{(N)}(m)(\tau,\mathfrak{z}):= m^{-1}\sum_{\substack{ad=m\\(a,N)=1\\ b \m d}} a^k\chi_N(a) \phi\left( \frac{a\tau + b}{d}, a\mathfrak{z} \right) \in J_{k,L(m)}^!(\Gamma_0(N),\chi_N).    
\end{equation}
The Fourier expansion of $\phi\Big|_{k,L} T_{-}^{(N)}(m)$ at infinity is given by
$$
\phi|_k T_{-}^{(N)}(m)(\tau,\mathfrak{z}) = \sum_{n, \ell} \sum_{\substack{a|(n,\ell,m)\\(a,N)=1\\a>0}} a^{k-1}\chi_N(a) f\left(\frac{nm}{a^2}, \frac{\ell}{a}\right) q^n\zeta^\ell,
$$
where $a|(n,\ell,m)$ means that $a|n$, $a|m$ and $a^{-1}\ell \in L'$. 

There is a standard isomorphism between Jacobi forms on $\SL_2(\ZZ)$ and modular forms for $\rho_L$:
\begin{equation}\label{eq:isomorphism-Jacobi}
\begin{split}
M_{k-\frac{1}{2}\rank(L)}^!(\rho_L) &\stackrel{\sim}{\longrightarrow} J_{k,L}^!(\SL_2(\ZZ)) \\
f(\tau)=\sum_{\gamma\in L'/L} f_\gamma(\tau)e_\gamma &\longmapsto \sum_{\gamma\in L'/L} f_\gamma(\tau) \Theta_{L,\gamma}(\tau,\mathfrak{z}).
\end{split}
\end{equation}
This map induces an isomorphism between $M_{k-\frac{1}{2}\rank(L)}(\rho_L)$ and $J_{k,L}(\SL_2(\ZZ))$. Thus a holomorphic Jacobi form of weight $\frac{1}{2}\rank(L)$ and index $L$ on $\SL_2(\ZZ)$ is a $\CC$-linear combination of Jacobi theta functions $\Theta_{L,\gamma}$. The weight $\frac{1}{2}\rank(L)$ is called the \textit{singular weight} of Jacobi forms. In \S \ref{sec:isomorphism} we will extend this classical isomorphism to an isomorphism between Jacobi forms on $\Gamma_1(N)$ and modular forms for $\rho_{U(N)\oplus L}$.

\section{Singular Borcherds products on maximal lattices}\label{sec:maximal}
In this section we classify singular Borcherds products on maximal lattices using the Jacobi forms representation of Borcherds products on $2U\oplus L$ established by Gritsenko and Nikulin \cite{GN98, Gri18}. We will prove 

\begin{theorem}\label{th:maximal}
$\II_{26,2}$ is the unique maximal even lattice of signature $(l,2)$ with $l\geq 3$ which has a holomorphic Borcherds product of singular weight, and Borcherds' form $\Phi_{12}$ is the unique singular Borcherds product on $\II_{26,2}$. 
\end{theorem}

We divide the proof into several lemmas. 

\begin{lemma}\label{lem:U}
Let $M$ be an even lattice. If $c \in M$ is a norm $0$ vector satisfying $(c, M)=\ZZ$, then we can split $M=U\oplus K$ with $c\in U$.
\end{lemma}
\begin{proof}
Since $(c, M)=\ZZ$, there exists $c' \in M$ such that $(c,c')=1$. We define $b=c' - \frac{1}{2}(c',c')c$. Then $b\in M$, $(b,b)=0$ and $(b,c)=1$. Let $K=M\cap c^\perp \cap b^\perp$. For any $v\in M$, $v-(v,b)c-(v,c)b \in K$. Thus $M=U\oplus K$ for $U=\ZZ c + \ZZ b$. 
\end{proof}

\begin{lemma}\label{lem:1-cusp}
Let $M$ be an even lattice of signature $(l,2)$ with $l\geq 3$ and $\Gamma$ be a finite-index subgroup of $\Orth^+(M)$. If $F$ is a nonzero modular form for $\Gamma$ of singular weight then there is a $1$-dimensional cusp at which $F$ does not vanish.
\end{lemma}
\begin{proof}
Let $c \in M$ be an isotropic vector and choose $c'\in M'$ with $(c,c')=1$. Since $F$ is nonzero, there is some vector $b \in M_{c,c'}\otimes\QQ$ whose Fourier coefficient in the expansion of $F$ on $\mathbb{H}_{c, c'}$ is nonzero. The vector $b$ must be isotropic because $F$ is singular; therefore, every zero-dimensional cusp ($\mathbb{Z}c$) is contained in a one-dimensional cusp ($\mathbb{Z}b + \mathbb{Z}c$). Since $F$ is singular, it is not a cusp form so it must fail to vanish identically on some $1$-dimensional cusp.
\end{proof}

\begin{lemma}\label{lem:maximal}
Let $M$ be a maximal even lattice. If $F$ is a modular form of singular weight on $\widetilde{\Orth}^+(M)$, then $M$ splits as $M=2U\oplus L$ in such a way that $F$ does not vanish at the $1$-dimensional cusp represented by $2U$, i.e. the associated zeroth Fourier--Jacobi coefficient is not zero.
\end{lemma}
\begin{proof}
By Lemma  \ref{lem:1-cusp}, there is some $1$-dimensional cusp, represented by an isotropic plane $P_F \subsetneq M$, on which $F$ does not vanish identically. Choose a primitive vector $c \in P_F$ and choose $c'\in M'$ with $(c,c')=1$. 
Since $M$ is maximal, $(c,M)=\ZZ$. By Lemma \ref{lem:U}, $M=U\oplus M_{c,c'}$ with $c\in U$. We take a primitive vector $b$ in $P_F \cap M_{c,c'}$. Then $M_{c,c'}$ is again maximal, which implies that $M_{c,c'}=U_1\oplus L$ with $b\in U_1$. By construction, $F$ does not vanish at the $1$-dimensional cusp associated to the decomposition $M=U\oplus U_1\oplus L$. 
\end{proof}

\begin{theorem}\label{th:Leech-type}
Let $L$ be an even positive-definite lattice and $F$ be a holomorphic Borcherds product of singular weight on $\widetilde{\Orth}^+(2U\oplus L)$ which does not vanish at the $1$-dimensional cusp represented by $2U$.  Then $L$ is a finite-index sublattice of the Leech lattice, and $F$ is the pullback of the Borcherds form $\Phi_{12}$. 
\end{theorem}
\begin{proof}
We interpret the preimage of $F$ under the Borcherds lift as a weakly holomorphic Jacobi form of weight $0$ and index $L$:
$$
\phi_0(\tau,\mathfrak{z}) =  \sum_{n\gg-\infty}\sum_{\ell \in L'} f(n,\ell) q^n\zeta^\ell \in J_{0,L}^!(\SL_2(\ZZ)).
$$
Note that $f(n,\ell)\in \ZZ$ for $2n-(\ell,\ell)<0$. 
By \cite[Theorem 4.2]{Gri18}, the $q^0$-term of $\phi_0$ is $[\phi_0]_{q^0}=\rank(L)$, i.e. $f(0,\ell)=0$ if $\ell\neq 0$; otherwise the zeroth Fourier--Jacobi coefficient would be zero. (In general, the leading non-zero Fourier--Jacobi coefficient of $F$ is given by the theta block $\widetilde{\psi}_{L;C}$ defined in the proof of \cite[Theorem 4.2]{Gri18}.) By \cite[Proposition 2.6]{Gri18}, we have
$$
\frac{\rank(L)}{24}= \sum_{n<0}\sum_{\ell \in L'} f(n,\ell) \sigma_1(n). 
$$
The number of the right hand side is an integer which we denote $N$.  Then $\rank(L)=24N$, and the zeroth Fourier--Jacobi coefficient of $F$ is $\Delta^N$, where $\Delta=\eta^{24}$ is the modular discriminant. By \cite[Corollary 4.3]{Gri18}, the first Fourier--Jacobi coefficient of $F$ is $-\Delta^N \phi_0$, which has to be a holomorphic Jacobi form of singular weight, trivial character and index $L$ on $\SL_2(\ZZ)$.  Considering the action of the `heat operator' $H_{12N}$ which annihilates singular Jacobi forms (see \cite[Lemma 2.2]{Wan21} for details), from $$0=H_{12N}(\Delta^N\phi_0)=\Delta^N H_0(\phi_0)$$ we obtain $$H(\phi_0)=-NE_2\phi_0.$$ Comparing $q^0$-terms in this equation yields 
\begin{equation}\label{eq:N=1}
0=\sum_{n=1}^N f(-n,0)\sigma_1(n)-N.    
\end{equation}
Since $\Delta^N\phi_0$ is a Jacobi form of singular weight, it can be expressed as a linear combination of Jacobi theta functions,
$$
\Delta^N\phi_0 = \sum_{\gamma} c_\gamma \Theta_{L,\gamma}, \quad c_{\gamma} \in \mathbb{C},
$$
where $\gamma$ runs through the cosets of $L'/L$ with $Q(\gamma)=0 \mod 1$.  Since all coefficients $f(-n,0)$ of $\phi_0$ come from $ \Theta_{L,0} / \Delta^N$, we conclude that $c_0$ is a positive integer, all $f(-n,0)$ are non-negative integers for $n\geq 0$ and $f(-N,0)=c_0$, which yields $N=1$ and $c_0=1$ by \eqref{eq:N=1}. Therefore,
$$
\phi_0 = \frac{ \Theta_{L,0}}{\Delta} + \sum_{\gamma} c_\gamma \frac{ \Theta_{L,\gamma}}{\Delta},  
$$
where $\Theta_{L,0}=1+O(q^2)$ and $\Theta_{L,\gamma}=O(q^2)$ if $\gamma\neq 0$ and $c_\gamma \neq 0$.
Since $F$ has only simple zeros (see \cite[Theorem 1.2]{WW23}), we find $c_\gamma \in\{0, 1\}$ for all $\gamma \neq 0$.
We define the set
$$
\mathcal{A}= \{ \gamma \in L'/L : Q(\gamma) =0 \m 1, \; c_\gamma=1 \}.
$$
The invariance of $\Delta\phi_0$ under $S = \begin{psmallmatrix} 0 & -1 \\ 1 & 0 \end{psmallmatrix}$ implies the equation
$$
\frac{1}{\sqrt{|L'/L|}}\sum_{\beta \in \mathcal{A}} e^{2\pi i(\beta, \gamma)} =1, \quad \text{for all $\gamma \in \mathcal{A}$}. 
$$
In particular, we have
\begin{enumerate}
    \item $|\mathcal{A}| = \sqrt{|L'/L|}$;
    \item $(\gamma, \beta)=0 \mod 1$, for any $\beta, \gamma \in \cA$. 
\end{enumerate}
Therefore, $L$ and $\mathcal{A}$ span an even unimodular lattice of rank $24$ without $2$-roots, which has to be the Leech lattice, and  $\Delta\phi_0$ is the pullback of the Jacobi theta function of the Leech lattice. This implies that $F$ is the pullback of $\Phi_{12}$.
\end{proof}

\begin{proof}[Proof of Theorem \ref{th:maximal}]
This follows from Lemma \ref{lem:maximal} and Theorem \ref{th:Leech-type}.
\end{proof}

\begin{remark}\label{rk:prime-symmetric}
In Theorem \ref{th:Leech-type}, if $F$ is modular for $\Orth^+(2U\oplus L)$ and $L$ is not unimodular, then $L$ is not of square-free level. Suppose $L$ has square-free level $N>1$.
The above set $\mathcal{A}$ is in fact a subgroup of $L'/L$. Since $N$ divides the order of $\mathcal{A}$, there exist elements of order $N$ in $\mathcal{A}$, and therefore $\mathcal{A}$ has elements of arbitrary order $n$ for $n|N$. Since any two elements of the same norm and order in $L'/L$ are conjugate under $\Orth(L'/L)$ (see e.g. \cite[Proposition 5.1]{Sch15}), $\mathcal{A}$ is the set of elements in $L'/L$ of norm $0\m 1$. By \cite[Propositions 3.1, 3.2, 3.3]{Sch06}, the number of elements in $L'/L$ of norm $0\m 1$ is not equal to $\sqrt{|L'/L|}$, which leads to a contradiction. 

In Theorem \ref{th:Leech-type}, if $F$ is modular for $\Orth^+(2U\oplus L)$, then $L$ is not necessarily the Leech lattice, for example $L$ can be a lattice in the genus of $2E_8\oplus A_1\oplus E_7$. 
\end{remark}

\section{Modular forms for the Weil representation and Jacobi forms}\label{sec:isomorphism}
In \cite{GN98, Gri18} Gritsenko and Nikulin expressed Borcherds products on $2U\oplus L$ in terms of Jacobi forms on $\SL_2(\ZZ)$ based on the isomorphism between modular forms for $\rho_L$ and Jacobi forms on $\SL_2(\ZZ)$. This expression has been used in the proof of Theorem \ref{th:maximal}. In this section we describe an isomorphism between modular forms for $\rho_{U(N)\oplus L}$ and sequences of Jacobi forms of level $\Gamma_1(N)$. We use this isomorphism to control the Fourier--Jacobi expansion of  Borcherds products on lattices of type $U(N)\oplus U\oplus L$. This will be the main tool to prove Theorems \ref{th:prime-level} and  \ref{th:moonshine}.

\subsection{An isomorphism between vector-valued modular forms and Jacobi forms}
Modular forms for the Weil representation associated to $U(N)$ can be described explicitly in terms of scalar-valued modular forms on $\Gamma_1(d)$, where $d$ runs through divisors of $N$. This idea appears as \cite[Lemma 2.6]{Bor95} and in full detail in \cite[Proposition 3.4]{Car12}. The isomorphism extends naturally to Jacobi forms of lattice index as follows.

Let $L$ be an even positive-definite lattice with quadratic form $Q$ and $N$ be a positive integer. Elements of the discriminant group of $U(N) \oplus L$ are written as tuples: 
$$
(a, b, \gamma), \quad a, b \in \mathbb{Z}/N\mathbb{Z}, \; \gamma \in L'/L,
$$ 
such that the quadratic form is $Q((a,b,\gamma))=ab/N+Q(\gamma)$.
The group ring is a tensor product $$\mathbb{C}[(U(N) \oplus L)' / (U(N) \oplus L)] = \mathbb{C}[\mathbb{Z}^2 / N \mathbb{Z}^2] \otimes_{\mathbb{C}} \mathbb{C}[L'/L]$$ with a natural basis of elements of the form $\mathfrak{e}_{(a, b)} \otimes \mathfrak{e}_{\gamma}$.

\begin{theorem}\label{th:iso}
Let $k\in \ZZ$. There is an isomorphism 
$$
\mathbb{J}:\; M^!_{k - \frac{1}{2}\rank(L)}(\rho_{U(N) \oplus L}) \stackrel{\sim}{\longrightarrow} \bigoplus_{d | N} J^!_{k, L}(\Gamma_1(N/d))
$$ 
which sends a vector-valued weakly holomorphic modular form 
$$
F(\tau) = \sum_{n \in \mathbb{Q}} \sum_{a, b \in \mathbb{Z}/N\mathbb{Z}} \sum_{\gamma \in L'/L} c_{(a, b), \gamma}(n) q^n \mathfrak{e}_{(a, b)} \otimes \mathfrak{e}_{\gamma}
$$ 
to the sequence of weakly holomorphic Jacobi forms 
$$
\phi_d(\tau, \mathfrak{z}) = \sum_{n \in \mathbb{Z}} \sum_{\ell \in L'}  \Big( \sum_{a \in \mathbb{Z}/N\mathbb{Z}} e^{2\pi i a d / N} c_{(a, 0), \ell +L}(n - Q(\ell)) \Big) q^n \zeta^{\ell}, \quad d | N.
$$
Moreover, $\mathbb{J}$ induces an isomorphism between the subspaces of holomorphic forms
$$
M_{k - \frac{1}{2}\rank(L)}(\rho_{U(N) \oplus L}) \stackrel{\sim}{\longrightarrow} \bigoplus_{d | N} J_{k, L}(\Gamma_1(N/d)).
$$ 
\end{theorem}

In the proof it is easier to work in the Fourier-transformed basis 
$$
\mathfrak{f}_{(a, b)} := \frac{1}{N} \sum_{c \in \mathbb{Z}/N\mathbb{Z}} e^{-2\pi i ac / N} \mathfrak{e}_{(c, b)}
$$
of $\mathbb{C}[\mathbb{Z}^2 / N\mathbb{Z}^2]$. The action of the Weil representation for $U(N)$ on this basis is straightforward to work out on the generators: 
$$
\rho_{U(N)}(T) \mathfrak{f}_{(a, b)} = \frac{1}{N} \sum_{c \in \mathbb{Z}/N\mathbb{Z}} e^{-2\pi i ac/N - 2\pi i bc/N} \mathfrak{e}_{(c, b)} = \mathfrak{f}_{(a + b, b)}
$$ 
and 
\begin{align*} 
\rho_{U(N)}(S) \mathfrak{f}_{(a, b)} &= \frac{1}{N} \sum_{c \in \mathbb{Z}/N\mathbb{Z}} e^{-2\pi i ac / N} \frac{1}{N} \sum_{\lambda, \mu \in \mathbb{Z}/N\mathbb{Z}} e^{2\pi i (\lambda b + \mu c) / N} \mathfrak{e}_{(\lambda, \mu)} \\ 
&= \frac{1}{N^2} \sum_{\lambda \in \mathbb{Z}/N\mathbb{Z}} e^{2\pi i \lambda b/N} \sum_{\mu \in \mathbb{Z}/N\mathbb{Z}} \Big( \sum_{c \in \mathbb{Z}/N\mathbb{Z}} e^{2\pi i c (\mu - a) / N} \Big) \mathfrak{e}_{(\lambda, \mu)} \\ 
&= \frac{1}{N} \sum_{\lambda \in \mathbb{Z}/N\mathbb{Z}} e^{2\pi i \lambda b/N} \mathfrak{e}_{(\lambda, a)} \\ 
&= \mathfrak{f}_{(-b, a)}
\end{align*}
and therefore (we view $v$ as a column vector.)
$$
\rho_{U(N)}(A) \mathfrak{f}_v = \mathfrak{f}_{Av}, \quad  \text{for all $A \in \mathrm{SL}_2(\mathbb{Z})$ and $v \in \mathbb{Z}^2 / N\mathbb{Z}^2$}.
$$
Taking the inverse Fourier transform, we have
$$
\mathfrak{e}_{(a, b)} = \sum_{c \in \mathbb{Z}/N\mathbb{Z}} e^{2\pi i ac / N} \mathfrak{f}_{(c, b)}.
$$

If the $F$ is decomposed into $\mathfrak{f}$-components $$F(\tau) = \sum_{a, b \in \mathbb{Z}/N\mathbb{Z}} \sum_{\gamma \in L'/L} f_{(a, b), \gamma}(\tau) \cdot \mathfrak{f}_{(a, b)} \otimes \mathfrak{e}_{\gamma},$$ then the Jacobi forms that make up $\mathbb{J}(F)$ are simply \begin{equation}\label{eq:ft-J}
\phi_d(\tau, \mathfrak{z}) = \sum_{\gamma \in L'/L} f_{(d, 0), \gamma}(\tau) \Theta_{L, \gamma}(\tau, \mathfrak{z})
\end{equation}
where $\Theta_{L, \gamma}(\tau, \mathfrak{z}) = \sum_{\ell \in L + \gamma} q^{Q(\ell)} e^{2\pi i ( \ell, \mathfrak{z} )}$ (as in Example \ref{ex:Theta}).

\begin{proof} 
(i) Define the linear map 
\begin{align*}
\mathrm{Tr}: \quad &\mathbb{C}[(U(N) \oplus L \oplus L)' / (U(N) \oplus L\oplus L)]  \rightarrow \mathbb{C}[\mathbb{Z}^2 / N\mathbb{Z}^2], \\
&\mathfrak{e}_{(a, b)} \otimes \mathfrak{e}_{\gamma} \otimes \mathfrak{e}_{\delta} \mapsto 
\begin{cases} 
\mathfrak{e}_{(a, b)}: & \gamma = \delta; \\ 0: & \text{otherwise}, \end{cases}
\end{align*}
for $a, b \in \mathbb{Z}/N\mathbb{Z}$ and $\gamma, \delta \in L'/L$. This map satisfies 
$$
\rho_{U(N)}(A)  \circ \mathrm{Tr} = \mathrm{Tr} \circ (\rho_{U(N) \oplus L} \otimes \bar{\rho}_L)(A), \quad A\in \mathrm{Mp}_2(\mathbb{Z}).
$$ 
This relation together with Example \ref{ex:Theta} shows that 
$$
\Phi(\tau, \mathfrak{z}) = \mathrm{Tr}(F \otimes \mathbf{\Theta}_L) = \sum_{n \in \mathbb{Q}} \sum_{a, b \in \mathbb{Z}/N\mathbb{Z}} \sum_{\ell \in L'} c_{(a, b), \ell+L}(n - Q(\ell)) q^n \zeta^{\ell} \mathfrak{e}_{(a, b)}
$$ 
is a vector-valued Jacobi form of weight $k$, index $L$ and multiplier $\rho_{U(N)}$ for $\SL_2(\ZZ)$: 
$$
\Phi\Big|_{k,L}A = \rho_{U(N)}(A)\Phi,  \quad A \in \SL_2(\ZZ).
$$

(ii) We express $\Phi$ in coordinates with respect to the $\mathfrak{f}$-basis:
\begin{align*} 
\Phi(\tau, \mathfrak{z}) &= \sum_{n \in \mathbb{Q}} \sum_{a, b, c \in \mathbb{Z}/N\mathbb{Z}} e^{2\pi i ac / N} \sum_{\ell \in L'} c_{(a, b), \ell+L}(n - Q(\ell)) q^n \zeta^{\ell} \mathfrak{f}_{(c, b)} \\ 
&= \sum_{c, b \in \mathbb{Z}/N\ZZ} \sum_{n \in \mathbb{Q}} \sum_{\ell \in L'} \Big( \sum_{a \in \mathbb{Z}/N\mathbb{Z}} e^{2\pi i ac / N} c_{(a, b), \ell+L}(n - Q(\ell)) \Big) q^n \zeta^{\ell} \mathfrak{f}_{(c, b)}.  
\end{align*} 
If $d | N$, then $\rho_{U(N)}(A) \mathfrak{f}_{(d, 0)}= \mathfrak{f}_{A(d,0)} = \mathfrak{f}_{(d, 0)}$ for all $A \in \Gamma_1(N/d)$, and therefore the $\mathfrak{f}_{(d, 0)}$-component $\phi_d$ of $\Phi$ is a Jacobi form of weight $k$, trivial character and index $L$ on $\Gamma_1(N/d)$. 

(iii) We have to show that the map $\mathbb{J}$ is a bijection. Suppose we are given a family $(\phi_d)_{d | N}$ of Jacobi forms of weight $k$ and index $L$, each $\phi_d$ of level $\Gamma_1(N/d)$. Since $(d, 0)$, $d | N$ represent the orbits of $\mathrm{SL}_2(\mathbb{Z})$ acting on $\mathbb{Z}^2 / N \mathbb{Z}^2$ (see also \cite[Lemma 3.2]{Car12}), we can construct a vector-valued Jacobi form $\Phi$ whose $\mathfrak{f}_{(d, 0)}$-component is $\phi_d$ in only way, namely by writing 
$$
\Phi(\tau, \mathfrak{z}) = \sum_{d | N} \sum_{A \in \mathrm{SL}_2(\mathbb{Z}) / \Gamma_1(N/d)} \Big( \phi_d \Big|_{k, L} A \Big) \mathfrak{f}_{A^{-1}(d, 0)}.
$$ 
Then for any $B \in \mathrm{SL}_2(\mathbb{Z})$ we have
\begin{align*} 
\Phi \Big|_{k, L} B (\tau, \mathfrak{z}) &= \sum_{d | N} \sum_{A \in \mathrm{SL}_2(\mathbb{Z}) / \Gamma_1(N/d)} \Big( \phi_d \Big|_{k, L} (A B) \Big) \mathfrak{f}_{A^{-1}(d, 0)} \\ 
&= \sum_{d | N} \sum_{A \in \mathrm{SL}_2(\mathbb{Z}) / \Gamma_1(N/d)} \Big( \phi_d \Big|_{k, L} A \Big) \mathfrak{f}_{BA^{-1}(d, 0)} \\ 
&= \rho_{U(N)}(B) \Phi(\tau, \mathfrak{z}). 
\end{align*} 
This uniquely determines the vector-valued modular form $F$ as the theta decomposition of $\Phi$. To be more precise, for any $a,b\in \ZZ$, there exists $A\in \SL_2(\ZZ)$ such that $A(a,b)=(d,0) \mod N$ for $d:=(a,b,N)$ (i.e. the greatest common divisor of $a$, $b$ and $N$).  We then write 
$$
\phi_d \Big|_{k,L} A (\tau,\mathfrak{z}) = \sum_{n,\ell} c_{d;A}(n,\ell) q^n\zeta^{\ell}. 
$$
For any $\gamma \in L'/L$ we fix a vector $\ell \in \gamma + L$ and define 
$$
f_{(a,b), \gamma}(\tau) = \sum_{n} c_{d;A}(n+Q(\ell),\ell)q^n. 
$$
Then the preimage of $(\phi_d)_{d|N}$ under $\mathbb{J}$ is
\[
F(\tau)=\sum_{a,b \in \ZZ / N\ZZ} \sum_{\gamma \in L'/L} f_{(a,b), \gamma}(\tau) \mathfrak{f}_{(a,b)}\otimes \mathfrak{e}_\gamma.  
\qedhere\]
\end{proof}

The group $(\mathbb{Z}/N\mathbb{Z})^{\times}$ acts on the discriminant form of $U(N)$ by automorphisms:
$$
u \cdot \mathfrak{e}_{(a, b)} = \mathfrak{e}_{(ua, u^{-1} b)}, \quad a,b \in \mathbb{Z}/N\mathbb{Z}, \quad u \in (\mathbb{Z}/N\mathbb{Z})^{\times}.
$$ 
The associated action on the $\mathfrak{f}$-basis is 
$$
u \cdot \mathfrak{f}_{(a, b)} = \frac{1}{N} \sum_{c \in \mathbb{Z}/N\mathbb{Z}} e^{-2\pi i ac/N} \mathfrak{e}_{(uc, u^{-1}b)} = \mathfrak{f}_{(u^{-1}a, u^{-1}b)}.
$$ 
We extend this action to modular forms for $\rho_{U(N) \oplus L}$ by defining 
$$
u \cdot F(\tau) = \sum_{n \in \mathbb{Q}} \sum_{a,b \in \mathbb{Z}/N\mathbb{Z}} \sum_{\gamma \in L'/L} c_{(a, b), \gamma}(n) q^n \mathfrak{e}_{(ua, u^{-1}b)} \otimes \mathfrak{e}_{\gamma}
$$ for
$$
F(\tau) = \sum_{n \in \mathbb{Q}} \sum_{a,b \in \mathbb{Z}/N\mathbb{Z}} \sum_{\gamma \in L'/L} c_{(a, b), \gamma}(n) q^n \mathfrak{e}_{(a, b)} \otimes \mathfrak{e}_{\gamma}.
$$
Let $M^!_{k}(\rho_{U(N) \oplus L})^{(\mathbb{Z}/N\mathbb{Z})^{\times}}$ denote the invariant subspace of $M^!_{k}(\rho_{U(N) \oplus L})$ with respect to the action of $(\ZZ/N\ZZ)^\times$. 
It is easy to check the following.
\begin{corollary}
The map $\mathbb{J}$ restricts to an isomorphism 
$$
M^!_{k - \frac{1}{2}\rank(L)}(\rho_{U(N) \oplus L})^{(\mathbb{Z}/N\mathbb{Z})^{\times}} \stackrel{\sim}{\longrightarrow} \bigoplus_{d | N} J^!_{k, L}(\Gamma_0(N/d)).
$$
\end{corollary}

We will now consider how the isomorphism $\mathbb{J}$ behaves with respect to the rationality of Fourier coefficients.

\begin{remark}
It is possible for $\mathbb{J}(F)$ to have an irrational Fourier expansion at infinity even if $F$ has only rational Fourier coefficients.  For example, $M=U(3)\oplus U\oplus A_2$ is a simple lattice and there exists a unique form $F \in M_{-1}^!(\rho_M)$ with rational Fourier expansion and principal part 
$$
q^{-1/3}\mathfrak{e}_{(1,0)}\otimes \mathfrak{e}_{\gamma} + q^{-1/3}\mathfrak{e}_{(-1,0)}\otimes \mathfrak{e}_{-\gamma} + 6\mathfrak{e}_{(0,0)}\otimes\mathfrak{e}_0,
$$
where $0\neq \gamma \in A_2'/A_2$.  If $\ell \in \gamma + A_2$ with $\ell^2=2/3$, then $q^0\zeta^\ell$ has coefficient $e^{2\pi i /3}$ in the Fourier expansion of $\phi_1$ at infinity. 

Conversely, it is possible for $\mathbb{J}^{-1}((\phi_d)_{d|N})$ to have irrational Fourier coefficients even if all $\phi_d$ have rational Fourier expansions at infinity. The lattice $M$ yields examples of this as well.
\end{remark}

\begin{lemma}\label{lem:rationality-1}
If $F\in M^!_{k - \rank(L)/2}(\rho_{U(N) \oplus L})^{(\mathbb{Z}/N\mathbb{Z})^{\times}}$ has integral Fourier expansion, then $\mathbb{J}(F)$ also has integral Fourier expansion at infinity.
\end{lemma}
\begin{proof}
For each $d|N$, the coefficient $\phi_d(n,\ell)$ of $q^n\zeta^\ell$ in the Fourier expansion of $\phi_d$ at infinity is 
\begin{align*}
&\sum_{a \in \mathbb{Z}/N\mathbb{Z}} e^{2\pi i a d / N} c_{(a, 0), \ell+L}(n - Q(\ell)) \\
=&\sum_{c|N} \sum_{\substack{(t,N/c)=1\\ t \m N/c}}e^{2\pi i tc d / N} c_{(t c, 0), \ell+L}(n - Q(\ell))\\
=&\sum_{c|N} c_{(c, 0), \ell+L}(n - Q(\ell)) \sum_{\substack{(t,N/c)=1\\ t \m N/c}}e^{2\pi i tc d / N},
\end{align*}
where we use $c_{(c,0),\ell}=c_{(tc,0),\ell}$ because there exists $u\in (\ZZ / N\ZZ)^\times $ such that $uc=tc \mod N$ and $u\cdot F = F$. Using the Ramanujan sum
$$
R(a,n):=\sum_{\substack{c\in \ZZ / n\ZZ\\ (c,n)=1}} e^{2\pi i ac/n} = \sum_{t|(a,n)} t \mu(n/t),
$$
where $\mu$ is the M\"obius function, we see that the coefficient $$\phi_d(n,\ell) = \sum_{c | N} \sum_{t | (d, N/c)} t \mu(N/ct) c_{(c,0), \ell+L}(n-Q(\ell))$$ is integral. 
\end{proof}

\begin{lemma}\label{lem:unique}
A form $F \in M^!_{k-\rank(L)/2}(\rho_{U(N)\oplus L})$ is uniquely determined by its components of type $\mathfrak{e}_{(a,0)}\otimes\mathfrak{e}_\gamma$ for all $a\in \ZZ/N\ZZ$ and $\gamma \in L'/L$. In particular, if all components $\mathfrak{e}_{(a,0)}\otimes \mathfrak{e}_\gamma$ have rational Fourier coefficients, then all Fourier coefficients of $F$ are rational. 
\end{lemma}
\begin{proof}
The first claim follows from Theorem \ref{th:iso}. Suppose all components of $\mathfrak{e}_{(a,0)} \otimes \mathfrak{e}_{\gamma}$ in $F$ have rational Fourier coefficients. By \cite{McG03}, the vector spaces $M_{k-\rank(L)/2}(\rho_{U(N)\oplus L})$ have bases with rational Fourier coefficients. It follows that $$M^!_{k-\rank(L)/2}(\rho_{U(N) \oplus L}) = \bigoplus_{n=0}^{\infty} \Delta^{-n} M_{k+12n - \rank(L)/2}(\rho_{U(N) \oplus L})$$ also has a basis with rational coefficients, say ${f_i}$. The expression of $F$ as a $\CC$-linear combination of these $f_i$ is determined by its components of type $\mathfrak{e}_{(a,0)}\otimes \mathfrak{e}_\gamma$; therefore, $F$ is a $\QQ$-linear combination of the $f_i$ and itself has rational Fourier coefficients. 
\end{proof}

\begin{lemma}\label{lem:rationality-2}
Let $(\phi_d)_{d|N} \in \bigoplus_{d|N}J_{k,L}^!(\Gamma_0(N/d))$. If every $\phi_d$ has rational Fourier expansion at infinity, then $\mathbb{J}^{-1}((\phi_d)_{d|N})$ also has rational Fourier expansion. 
\end{lemma}
\begin{proof}
By Lemma \ref{lem:unique}, it suffices to prove that all components $\mathfrak{e}_{(a,0)}\otimes \mathfrak{e}_\gamma$ of $\mathbb{J}^{-1}((\phi_d)_{d|N})$ have rational Fourier expansions. We find that the Fourier coefficient of such a component is
\begin{align*}
c_{(a,0), \gamma}(n) &= \frac{1}{N}\sum_{c \m N} e^{-2\pi iac/N} f_{(c,0), \gamma}(n) \\
&= \frac{1}{N}\sum_{c \m N} e^{-2\pi iac/N} \left(\phi_{(c,N)}\Big|_{k,L}A_c\right)(n+Q(\ell), \ell)\\
&= \frac{1}{N}\sum_{c \m N} e^{-2\pi iac/N} \phi_{(c,N)}(n+Q(\ell), \ell)
\end{align*}
\begin{align*}
&= \frac{1}{N}\sum_{d|N} \phi_d(n+Q(\ell),\ell) \sum_{\substack{(c,N/d)=1\\ c \m N/d}} e^{-2\pi iacd/N}\\
&= \frac{1}{N}\sum_{d|N} \phi_d(n+Q(\ell),\ell) R(a,N/d) \in \QQ.
\end{align*}
Here, $\ell \in \gamma +L$ is any vector and $\phi(n,\ell)$ denotes the coefficient of $q^n\zeta^\ell$ in the Fourier expansion of $\phi$ at infinity, and $A_c \in \Gamma_0(N/d)$ is any matrix satisfying $A_c(c,0)=(d,0) \m N$ for $d=(c,N)$.
\end{proof}

\begin{remark}
In Lemma \ref{lem:rationality-2}, it is possible for $\mathbb{J}^{-1}((\phi_d)_{d|N})$ to have a non-trivial denominator even if $(\phi_d)_{d|N}$ has integral Fourier expansion at every cusp. For example, let $N=p$ be prime. Suppose $\phi_1 \in J_{k,L}^!(\SL_2(\ZZ))$ and $\phi_p=0$. Then the coefficient of $q^0\mathfrak{e}_{(0,0)}\otimes\mathfrak{e}_0$ in the Fourier expansion of $\mathbb{J}^{-1}((\phi_d)_{d|N})$ is $\frac{p-1}{p}\phi_1(0,0)$.
\end{remark}

Vector-valued modular forms $F_{\Gamma_0(N), f, 0}$ constructed by averaging a scalar modular form $f$ as in Theorem \ref{th:scalar-to-vector} correspond under $\mathbb{J}$ to sequences $(\phi_d)_{d|N}$ in which each $\phi_d$, $d>1$ is the trace of $\phi_1$ over $\Gamma_0(N/d) / \Gamma_0(N)$:

\begin{lemma}\label{lem:special-input}
Suppose $M=U(N_1)\oplus U \oplus L$ has level $N$. If $f$ is a scalar-valued weakly holomorphic modular form of weight $k-\rank(L)/2$ and character $\chi_{D_M}$ on $\widetilde{\Gamma}_0(N)$,  then 
$\mathbb{J}(F_{\Gamma_0(N), f, 0}) = (\phi_d)_{d|N_1}$ is 
\begin{align*}
\phi_d(\tau,\mathfrak{z})&=\sum_{A\in \Gamma_0(N_1/d) / \Gamma_0(N)} \psi_1\Big|_{k,L}A (\tau,\mathfrak{z}), \quad d|N_1,
\end{align*}
where
$$
\psi_1(\tau,\mathfrak{z})=f(\tau)\Theta_{L,0}(\tau, \mathfrak{z}).
$$
\end{lemma}
\begin{proof}
Obviously, $N_1|N$. 
Let $\kappa=k-\rank(L)/2$. The form $F := F_{\Gamma_0(N), f, 0}$ is 
\begin{align*}
F(\tau) &= \sum_{A \in \widetilde{\Gamma}_0(N) \backslash \mathrm{Mp}_2(\mathbb{Z})} \Big(f \Big|_\kappa A\Big) \rho_{U(N)}(A^{-1}) \mathfrak{e}_{(0, 0)} \otimes \rho_{U\oplus L}(A^{-1}) \mathfrak{e}_0 \\ 
&= \sum_{d \in \mathbb{Z}/N_1\mathbb{Z}} \sum_{A \in \widetilde{\Gamma}_0(N) \backslash \mathrm{Mp}_2(\mathbb{Z})} \Big(f \Big|_\kappa A\Big)  \mathfrak{f}_{A^{-1} \begin{psmallmatrix} d \\ 0 \end{psmallmatrix}} \otimes \rho_{U \oplus L}(A^{-1}) \mathfrak{e}_0. 
\end{align*} 
For any divisor $d | N_1$, the $\mathfrak{f}_{(d, 0)}$-component of $F$ is 
$$
\sum_{A \in \Gamma_0(N_1/d) / \Gamma_0(N)} \Big(f \Big|_\kappa A\Big) \rho_{U \oplus L}(A^{-1}) \mathfrak{e}_0 = \sum_{\gamma \in L'/L} \sum_{A \in \Gamma_0(N_1/d) / \Gamma_0(N)} \Big(f \Big|_\kappa A\Big) \langle \mathfrak{e}_{\gamma}, \rho_{U \oplus L}(A^{-1}) \mathfrak{e}_0 \rangle \cdot \mathfrak{e}_{\gamma},
$$ 
so after identifying the Weil representations of $U \oplus L$ and $L$ and using \eqref{eq:ft-J} we find that the associated Jacobi form is 
$$
\phi_d(\tau, \mathfrak{z}) = \sum_{A \in \Gamma_0(N_1/d) / \Gamma_0(N)} f \Big|_\kappa A \sum_{\gamma \in L'/L} \langle \mathfrak{e}_{\gamma}, \rho_L^{-1}(A) \mathfrak{e}_0 \rangle \Theta_{L, \gamma}(\tau, \mathfrak{z}).
$$ 
Using the Poisson summation formula one finds 
$$
\sum_{\gamma \in L'/L} \langle \mathfrak{e}_{\gamma}, \rho_L^{-1}(A) \mathfrak{e}_0 \rangle \Theta_{L, \gamma}(\tau, \mathfrak{z}) = \Theta_{L, 0}\Big|_{\rank(L)/2, L} A(\tau, \mathfrak{z})
$$ 
for $A = S, T$  and therefore for arbitrary $A \in \mathrm{Mp}_2(\mathbb{Z})$, so we obtain the expression for $\phi_d$.
\end{proof}

\subsection{The Fourier--Jacobi expansion of Borcherds products}
We will describe the Fourier--Jacobi expansion of a Borcherds product on $M=U(N)\oplus U\oplus L$ at the level $N$ cusp. Let $F$ be a weakly holomorphic modular form of weight $-\rank(L)/2$ for $\rho_{M}$ with integral principal part. We write its Fourier expansion as
\begin{align*}
F(\tau)&=\sum_n\sum_{a,b \in \ZZ / N\ZZ} \sum_{\gamma \in L'/L} c_{(a,b),\gamma}(n)q^n \mathfrak{e}_{(a,b)} \otimes \mathfrak{e}_\gamma\\ &=\sum_n \sum_{a,b \in \ZZ / N\ZZ} \sum_{\gamma \in L'/L} f_{(a,b), \gamma}(n)q^n \mathfrak{f}_{(a,b)} \otimes \mathfrak{e}_\gamma.   
\end{align*}
For any $a\in \NN$, define 
\begin{equation}
\phi_a(\tau,\mathfrak{z}) = \sum_{n,\ell} f_{(a,0),\ell}(n-Q(\ell))q^n\zeta^\ell   
\end{equation}
and denote its Fourier expansion at infinity by $\sum_{n,\ell}\phi_a(n,\ell)q^n\zeta^\ell$. 
Clearly, $(\phi_d)_{d|N}$ is the image of $F$ under $\mathbb{J}$. We have the following relation:
$$
\phi_a = \phi_{d}\Big|_{0,L}A_{a}, \quad d=(a,N),
$$
where $A_{a}\in \Gamma_0(N/d)$ is any matrix satisfying
$$
A_{a} = \left( \begin{array}{cc}
(a/d)^{-1} & 0 \\ 
0 & a/d
\end{array}   \right) \mod N/d.  
$$
For any $d|N$ we decompose 
$$
\phi_d = \sum_{\chi \m N/d} \phi_{d,\chi}
$$
into a sum of Jacobi forms on $\Gamma_0(N/d)$ with Dirichlet characters of modulus $N/d$.  

In \S \ref{subsec:Borcherds-products}, fix the class $\frac{e}{N}+U(N)$ as $(1, 0)$ and $e'+U(N)$ as $(0, 1)$. Then the Borcherds multiplicative lift of $F$ on $U(N)\oplus U\oplus L$ has the Fourier expansion
\begin{align*}
\Borch(F)(Z)&=e^{2\pi i(\rho, Z)} \prod_{ \substack{\lambda\in U\oplus L'\\ \lambda>0}}\prod_{a \m N}(1-e^{2\pi i a/N}e^{2\pi i(\lambda,Z)})^{c_{(a,0),\lambda}(-\lambda^2/2)}\\
&=q^A\zeta^B s^C\prod_{(n,\ell, m)>0} \prod_{a \m N}(1-e^{2\pi i a/N}q^n\zeta^\ell s^m)^{c_{(a,0),\ell}(nm-\ell^2/2)}
\end{align*}
where $Z=(\omega, \mathfrak{z},\tau)\in \HH \times (L\otimes\CC) \times \HH$, $\rho=(A,B,C)$ is the Weyl vector, $(n,\ell,m)\in \ZZ\oplus L' \oplus \ZZ$, $q=e^{2\pi i\tau}$, $\zeta^\ell = e^{2\pi i(\ell,\mathfrak{z})}$, $s=e^{2\pi i\omega}$, and $(n,\ell,m)>0$ means that either $m>0$, or $m=0$ and $n>0$, or $m=n=0$ and $\ell<0$.
We decompose $\Borch(F)$ into two parts:
\begin{align*}
\Borch_0(F)(Z)&= q^A\zeta^B s^C\prod_{(n,\ell)>0} \prod_{a \m N}(1-e^{2\pi i a/N}q^n\zeta^\ell)^{c_{(a,0),\ell}(-\ell^2/2)},\\
\Borch_1(F)(Z)&= \prod_{\substack{(n,\ell,m)>0\\m>0}} \prod_{a \m N}(1-e^{2\pi i a/N}q^n\zeta^\ell s^m)^{c_{(a,0),\ell}(nm-\ell^2/2)}.
\end{align*}
After substituting $c_{(a,0),\lambda}=\frac{1}{N}\sum_{c\m N} e^{-2\pi i ac/N} f_{(c,0),\lambda}$, the logarithm of $\Borch_1(F)$ becomes
\begin{align*} &\log \, \Borch_1(F)(Z) \\
=&\sum_{\substack{(n,\ell,m)>0\\ m>0}} \sum_{a\m N} \frac{1}{N}\sum_{c \m N} e^{-2\pi i ac/N} f_{ (c,0),\ell}(nm-\ell^2/2) \log(1-e^{2\pi i a/N}q^n\zeta^\ell s^m) \\
=& -\sum_{\substack{(n,\ell,m)>0\\ m>0}} \sum_{a \m N} \frac{1}{N}\sum_{c \m N} e^{-2\pi i ac/N} f_{ (c,0),\ell}(nm-\ell^2/2) \sum_{t=1}^\infty t^{-1} e^{2\pi i at/N}(q^n\zeta^\ell s^m)^t \\
= & -\sum_{\substack{(n,\ell,m)>0\\ m>0}} \sum_{c\m N}  f_{(c,0),\ell}(nm-\ell^2/2) \sum_{t=1}^\infty t^{-1} (q^n\zeta^\ell s^m)^t \sum_{a\m N} \frac{1}{N} e^{2\pi i a(t-c)/N}\\
= & -\sum_{\substack{(n,\ell,m)>0\\ m>0}}  \sum_{t=1}^\infty f_{(t,0),\ell}(nm-\ell^2/2) t^{-1} (q^n\zeta^\ell s^m)^t \\
=& -\sum_{\substack{(n,\ell,m)>0\\ m>0}}\sum_{a=1}^\infty \phi_{a}(nm,\ell) a^{-1} (q^n\zeta^\ell s^m)^a\\
= &- \sum_{\substack{(n,\ell,m)>0\\ m>0}} \sum_{d|N} \sum_{\chi \m N/d} d^{-1} \sum_{(a,N/d)=1} a^{-1} \chi(a) \phi_{d,\chi}(nm,\ell) (q^{dn}\zeta^{d\ell}s^{dm})^a\\
=& -\sum_{d|N} \sum_{\chi \m N/d} d^{-1} \sum_{\substack{(n, \ell,m)>0\\m>0}} \Big(\sum_{\substack{a|(n,\ell,m)\\ (a, N/d)=1}} a^{-1} \chi(a) \phi_{d,\chi}(nm/a^2,\ell/a)\Big) (q^{n}\zeta^{\ell}s^{m})^d\\
=& - \sum_{d|N} \sum_{\chi \m N/d} \Grit(d^{-1} \phi_{d,\chi})(dZ), 
\end{align*}
where for $\phi\in J_{0,L}^!(\Gamma_0(t), \chi_t)$ the formal additive lift $\Grit(\phi)$ of weight $0$ is defined as
\begin{equation}
\Grit(\phi)(Z) = \sum_{m=1}^{\infty} \phi\Big|_{0,L} T_{-}^{(t)}(m)(\tau,\mathfrak{z}) s^m.
\end{equation}
(Here $T_{-}^{(t)}(m)$ was defined in \eqref{eq:index-raising}.)
In particular, 
\begin{equation}\label{eq:b1}
    \Borch_1(F)(Z) = \exp\left( - \sum_{d|N} \sum_{\chi \m N/d} \Grit(d^{-1} \phi_{d,\chi})(dZ) \right).
\end{equation}

Similarly, we have
\begin{equation}
    \Borch_0(F)(Z) =q^A\zeta^B s^C\cdot \exp\left( - \sum_{d|N} \sum_{\chi \m N/d} \Grit_0(d^{-1} \phi_{d,\chi})(dZ) \right), 
\end{equation}
where for $\phi\in J_{0,L}^!(\Gamma_0(t), \chi_t)$ the formal Hecke operator $T_{-}^{(t)}(0)$ of index $0$ is defined as
\begin{equation}
\Grit_0(\phi)(\tau, \mathfrak{z})=  \sum_{(n,\ell)>0} \sum_{\substack{(a,t)=1\\ a|(n,\ell)}} a^{-1} \chi(a)\phi(0, \ell/a) q^{n}\zeta^\ell.      
\end{equation}

The above also implies the simple expression 
\begin{equation}\label{eq:exp}
\Borch(F)(Z) = q^A\zeta^B s^C \prod_{(n,\ell,m)>0} \exp\left(- \sum_{a=1}^\infty \phi_{a}(nm,\ell) \frac{(q^n\zeta^\ell s^m)^a}{a}  \right).     
\end{equation}

The modularity of $\Borch_1(F)$ under $\Gamma_1(N)\ltimes H(L)$ is clear from Expression \eqref{eq:b1}. Therefore, 
\begin{equation}
\vartheta_F := \Borch_0(F)(Z)\cdot s^{-C}    
\end{equation}
defines a meromorphic Jacobi form of index $L(C)$ on $\Gamma_1(N)$ with some multiplier.  The Fourier--Jacobi expansion of $\Borch(F)$ begins
\begin{equation}
    \Borch(F) = \vartheta_F \cdot s^C - \vartheta_F \phi_1 \cdot s^{C+1} + O(s^{C+2}). 
\end{equation}

The weight of $\Borch(F)$, i.e. half of the coefficient of the $q^0\mathfrak{e}_{(0,0)}\otimes\mathfrak{e}_{0}$-term in $F$ is given by
\begin{equation}\label{eq:weight}
\begin{split}
\frac{1}{2N}\sum_{c\m N} f_{(c,0),0}(0) &= \frac{1}{2N} \sum_{c\m N} \phi_{c}(0,0) \\
&= \frac{1}{2N} \sum_{d|N} \sum_{\chi \m N/d} \phi_{d,\chi}(0,0) \sum_{ c\m N/d }\chi(c)\\
&= \frac{1}{2N}\sum_{d|N} \varphi(N/d)\phi_{d,\chi_0}(0,0), 
\end{split}
\end{equation}
where $\varphi(n)$ is Euler's totient function and $\chi_0$ is the principal Dirichlet character.

Let $\lambda=(n,\ell,m) \in U\oplus L'$ be a primitive vector of positive norm. Then the multiplicity of $\lambda^\perp$ in the divisor of $\Borch(F)$ equals
\begin{equation}
\begin{split}
&\frac{1}{N} \sum_{c \m N} \sum_{a=1}^\infty f_{(c,0),a\lambda}(-a^2\lambda^2/2)\\
=&\frac{1}{N}\sum_{d|N}\sum_{\chi \m N/d} \sum_{a=1}^\infty \phi_{d,\chi}(a^2nm,a\ell) \sum_{c\m N/d} \chi(c)\\
=& \frac{1}{N}\sum_{d|N}  \sum_{a=1}^\infty \varphi(N/d) \phi_{d,\chi_0}(a^2nm,a\ell).
\end{split}
\end{equation}

\begin{theorem}\label{th:FJ-level-N}
Suppose $F\in M^!_{-\rank(L)/2}(\rho_{U(N) \oplus L})^{(\mathbb{Z}/N\mathbb{Z})^{\times}}$ has integral principal part. At the $1$-dimensional cusp determined by $U(N)\oplus U$, the product $\Borch(F)$ has the expressions
\begin{align}
\Borch(F)&=q^A\zeta^B s^C \prod_{(n,\ell,m)>0} \left( 1-q^n\zeta^\ell s^m\right)^{\mathrm{mult}(n,\ell,m)} \label{eq:B1} \\
&=q^A\zeta^B s^C \prod_{d|N} \prod_{(n,\ell,m)>0} \left( 1-\big(q^n\zeta^\ell s^m\big)^d\right)^{\mathrm{mult}_d(nm,\ell)} \label{eq:B2}
\end{align}
where 
\begin{equation}
\mathrm{mult}(n,\ell,m)=\sum_{\substack{b,d>0\\bd|(n,\ell,m,N)}} \frac{\mu(b)}{bd}\phi_d\Big( \frac{nm}{b^2d^2}, \frac{\ell}{bd} \Big) \end{equation}
and 
\begin{equation}
\mathrm{mult}_d(nm,\ell)=\sum_{t|d} \frac{\mu(d/t)}{d}\phi_t(nm,\ell)
\end{equation}
and
\begin{align}
&A=\frac{1}{24}\sum_{\ell} \phi_N(0,\ell), \quad B=\frac{1}{2}\sum_{\ell>0}\phi_N(0,\ell)\ell, \\
&C=\frac{1}{\rank(L)}\sum_{\ell>0} \phi_N(0,\ell)(\ell,\ell) = A - \sum_{n>0}\sum_{\ell}\phi_N(-n,\ell)\sigma_1(n).      
\end{align}
Moreover, $\Borch(F)$ has the Fourier--Jacobi expansion
\begin{equation}
\Borch(F)(Z) = \vartheta_F(\tau,\mathfrak{z}) \cdot s^C \cdot \exp\Big(-\sum_{d|N}\Grit\big(d^{-1}\phi_d\big)(dZ)\Big),
\end{equation}
where 
\begin{equation}
\vartheta_F(\tau,\mathfrak{z}) = \prod_{d|N} \left( \eta(d\tau)^{\mathrm{mult}_d(0,0)}\prod_{\ell>0} \left( \frac{\vartheta(d\tau, d(\ell,\mathfrak{z}))}{\eta(d\tau)} \right)^{\mathrm{mult}_d(0,\ell)} \right)
\end{equation}
and
\begin{align}
\eta(\tau)&=q^{\frac{1}{24}}\prod_{n=1}^\infty(1-q^n),\\
\vartheta(\tau,z)&=q^{\frac{1}{8}}(e^{\pi iz} - e^{-\pi iz})\prod_{n=1}^\infty (1-q^ne^{2\pi iz})(1-q^ne^{-2\pi iz})(1-q^n), \quad z\in \CC.
\end{align}
\end{theorem}
\begin{proof}
By the previous calculation, the logarithm of $\Borch(F)\cdot q^{-A}\zeta^{-B} s^{-C}$ has the form
\begin{align*}
&- \sum_{(n,\ell,m)>0} \sum_{d|N} d^{-1} \sum_{\substack{a>0\\(a,N/d)=1}} a^{-1} \phi_{d}(nm,\ell) (q^{dn}\zeta^{d\ell}s^{dm})^a   \\
=& \sum_{(n,\ell,m)>0} \sum_{d|N} d^{-1} \phi_d(nm,\ell) \sum_{b|N/d} \frac{\mu(b)}{b}\log\Big(1 - (q^{dn}\zeta^{d\ell}s^{dm})^b \Big),
\end{align*}
where the equality uses M\"obius inversion, through the identity 
\begin{equation}\label{eq:log}
\sum_{\substack{n>0\\ (n,N)=1}} \frac{x^n}{n} = - \sum_{b|N} \frac{\mu(b)}{b}\log(1-x^b).    
\end{equation}
Taking exponentials proves \eqref{eq:B1} and \eqref{eq:B2}. It remains to determine $\vartheta_F$ and the Weyl vector $(A,B,C)$. We calculate (recalling that $(n,\ell)>0$ means either $n>0$ or $n=0$ and $\ell<0$):
\begin{align*}
&\prod_{d|N} \prod_{(n,\ell)>0} \left( 1-(q^n\zeta^\ell)^d\right)^{\mathrm{mult}_d(0,\ell)}  \\
=&\prod_{d|N} \left[ \prod_{n=1}^\infty\Big(1-q^{dn}\Big)^{\mathrm{mult}_d(0,0)} \cdot  \prod_{\ell>0}\left((1-\zeta^{-d\ell})\prod_{n=1}^\infty\Big(1-q^{dn}\zeta^{d\ell}\Big)\Big(1-q^{dn}\zeta^{-d\ell}\Big)\right)^{\mathrm{mult}_d(0,\ell)} \right]\\
=& \prod_{d|N}\left[ \Big( q^{-\frac{d}{24}}\eta(d\tau)\Big)^{\mathrm{mult}_d(0,0)} \cdot \prod_{\ell>0} \left( q^{-\frac{d}{12}}\zeta^{-\frac{d\ell}{2}} \frac{\vartheta(d\tau, d(\ell,\mathfrak{z}))}{\eta(d\tau)} \right)^{\mathrm{mult}_d(0,\ell)} \right]\\
=&q^{-\hat{A}}\zeta^{-\hat{B}} \cdot  \prod_{d|N} \left( \eta(d\tau)^{\mathrm{mult}_d(0,0)}\prod_{\ell>0} \left( \frac{\vartheta(d\tau, d(\ell,\mathfrak{z}))}{\eta(d\tau)} \right)^{\mathrm{mult}_d(0,\ell)} \right),
\end{align*}
where
\begin{align*}
\hat{A}=&\frac{1}{24}\sum_{d|N}\sum_{\ell}d\cdot \mathrm{mult}_d(0,\ell)
\\ =& \frac{1}{24}\sum_{d|N}\sum_{\ell} \sum_{t|d} \mu(d/t) \phi_t(0,\ell)\\
=& \frac{1}{24} \sum_{t|N} \sum_{\ell} \phi_t(0,\ell) \sum_{d|N/t} \mu(d)
= \frac{1}{24} \sum_{\ell} \phi_N(0,\ell)
\end{align*}
and similarly 
\begin{align*}
\hat{B}=&\frac{1}{2}\sum_{d|N}\sum_{\ell>0} d\cdot\mathrm{mult}_d(0,\ell)\ell
=\frac{1}{2} \sum_{\ell>0} \phi_N(0,\ell)\ell. 
\end{align*}
Comparing leading coefficients shows that $A = \hat{A}$ and $B = \hat{B}$. Since $\vartheta_F$ is a meromorphic Jacobi form of index $L(C)$ on $\Gamma_0(N)$, the number $C$ is determined by the relation
$$
\sum_{d|N} \sum_{\ell>0} d \cdot \mathrm{mult}_d(0,\ell)\cdot (\ell, \mathfrak{z})^2 = C(\mathfrak{z}, \mathfrak{z}),
$$
which reduces to 
$$
\sum_{\ell>0} \phi_N(0,\ell) \cdot (\ell, \mathfrak{z})^2 = C(\mathfrak{z}, \mathfrak{z}). 
$$
Finally \cite[Proposition 2.6]{Gri18} leads to the claimed formula for $C$ because $\phi_N\in J_{0,L}^!(\SL_2(\ZZ))$. 
\end{proof}

In Theorem \ref{th:FJ-level-N}, the weight of $\Borch(F)$ equals $\frac{1}{2}\sum_{d|N}\mathrm{mult}_d(0,0)$, which agrees with  \eqref{eq:weight}. 

\begin{remark}
In general, the leading Fourier--Jacobi coefficient of $\Borch(F)$ can be expressed as
$$
\vartheta_F(\tau,\mathfrak{z})= \eta(\tau)^{2k} \cdot \eta_F(\tau) \cdot \prod_{\ell>0}\prod_{a\m N}\left(\frac{\vartheta(\tau, (\ell,\mathfrak{z})+a/N)}{\eta(\tau)}\right)^{c_{(a,0),\ell}(-\ell^2/2)},
$$
where $k=\frac{1}{2}c_{(0,0),0}(0)$ is the weight of $\Borch(F)$ and 
$$
\eta_F(\tau) = q^{A'} \prod_{n=1}^\infty \prod_{a=1}^{N-1}\Big( 1-q^n e^{2\pi ia/N} \Big)^{c_{(a,0),0}(0)} 
$$
defines a meromorphic modular form of weight $0$ on $\Gamma_1(N)$ for some $A'\in \QQ$. This expression determines the values of $B$ and $C$ in the Weyl vector of $\Borch(F)$, and the invariance of $\Borch(F)$ under exchanging $\tau$ and $\omega$ yields the value of $A$. In fact, the vector $(A,B,C)$ is completely determined by the component $\phi_N$ as in Theorem \ref{th:FJ-level-N}. Thus $A'=\frac{1}{24}\sum_{a=1}^{N-1} c_{(a,0),0}(0)$. 
\end{remark}

\begin{remark}
An analogue of Theorem \ref{th:iso} for $\Orth(1,2)$ was recently established in \cite{DHR22}.
\end{remark}

\begin{remark}
All results in this section continue to hold when $L=0$, in which case weakly holomorphic Jacobi forms $\phi_d$ of weight $0$ on $\Gamma_1(N/d)$ degenerate to scalar-valued weakly holomorphic modular forms of weight $0$ on $\Gamma_1(N/d)$. This was treated in \cite{Car12}. The Weyl vector in Theorem \ref{th:FJ-level-N} simplifies to 
$$
(A,C)=\Big( \frac{1}{24} \phi_N(0), \; A - \sum_{n>0}\phi_N(-n)\sigma_1(n) \Big).
$$
As an application, let $b_d \in \ZZ$ for $d|N$.  For any $d|N$ we define $\phi_d$ to be the integer $c_d=\sum_{t|d} tb_t$. Then \eqref{eq:B2} in Theorem \ref{th:FJ-level-N} yields
\begin{align*}
    \Borch(F)(\omega,\tau)&=q^A s^C\prod_{d|N} \prod_{m=1}^\infty \Big(1-s^{md}\Big)^{\mathrm{mult}_d(0,0)} \prod_{n=1}^\infty \Big(1-q^{nd}\Big)^{\mathrm{mult}_d(0,0)} \\
    &= \prod_{d|N} \eta(d\tau)^{b_d}\eta(d\omega)^{b_d}. 
\end{align*}
In other words, for an eta quotient $\eta_g$ associated to any cycle shape $g=\prod_{d|N}d^{b_d}$ the form $\eta_g(\tau)\eta_g(\omega)$ is realized as above as a Borcherds product on $U(N)\oplus U$. By \cite{Zem20}, its pullback along the embedding $U(N)\oplus\latt{-2}<U(N)\oplus U$ is again a Borcherds product. In other words, every eta quotient $f(\tau) = \prod_{k=1}^{\infty} \eta( m_k \tau)^{b_k}$ can be realized as a Borcherds product on the lattice $U(N) \oplus \latt{-2}$ with $N = \mathrm{lcm}(m_k: \, b_k \ne 0)$.
\end{remark}

\begin{remark}
Let $N_1$ be a positive integer. The function $\Borch(F)$ in Theorem \ref{th:FJ-level-N} can be viewed as a Borcherds product on $U(NN_1)\oplus U\oplus L$ and by \eqref{eq:exp} we find that the associated input as Jacobi forms $(\hat{\phi}_d)_{d|NN_1}$ is 
$$
\hat{\phi}_d = \phi_{(d,N)}, \quad d|NN_1. 
$$
\end{remark}

\begin{remark}
Aoki and Ibukiyama \cite[\S 6]{AI05} constructed Siegel modular forms of degree two as infinite products in terms of Jacobi forms of index $\latt{2}$ on $\Gamma_0(N)$ under certain assumptions. Later, Cl\'{e}ry and Gritsenko \cite[Theorem 3.1]{CG11} removed the assumptions and extended the construction to Siegel paramodular forms of degree two which can be realized as modular forms on lattices of type $U(N)\oplus U\oplus \latt{2t}$. Our approach further extends their construction to more general lattices of type $U(N)\oplus U\oplus L$. Note that Cl\'ery--Gritsenko constructed infinite products associated to single Jacobi forms of higher level. The input of \cite[Theorem 3.1]{CG11} in our context has the form 
\begin{align*}
&\phi_1(\tau,z) \in J_{0,\latt{2t}}^!(\Gamma_0(N)),\\
&\phi_d(\tau,z) = \sum_{A\in \Gamma_0(N/d) / \Gamma_0(N)} \phi_1\Big|_{0,\latt{2t}}A(\tau,z),\quad  d|N.
\end{align*}

More importantly, our approach shows that their infinite products in terms of Jacobi forms are in fact true Borcherds products. This allows us to prove that their infinite products have meromorphic continuations to the entire Siegel upper half space $\HH_2$ and to calculate the multiplicity of each divisor (including some divisors which are not visible in the infinite product expansion).
\end{remark}

\begin{remark}
Let $(\phi_d)_{d|N} \in \bigoplus_{d|N}J_{k,L}^!(\Gamma_1(N))$ for a positive integer $k$. The Borcherds additive lift \cite[Theorem 14.3]{Bor98} sends $F:=\mathbb{J}^{-1}((\phi_d)_{d|N})$ to a meromorphic modular form of weight $k$ and trivial character on $\widetilde{\Orth}^+(M)$ whose singularities are poles of order $k$ along certain rational quadratic divisors determined by the principal part of $F$. A computation analogous to the proof of Theorem \ref{th:FJ-level-N} shows that this form has Fourier--Jacobi expansion on $U(N)\oplus U\oplus L$ as
\begin{equation}
\Grit(F)(Z) = \sum_{d|N} \sum_{\chi \m N/d} \Grit(d^{k-1} \phi_{d,\chi})(dZ),   
\end{equation}
where for $\phi \in J_{k,L}^!(\Gamma_0(t),\chi_t)$ the Gritsenko additive lift $\Grit(\phi)$ is
\begin{equation}
\Grit(\phi)(Z) = \sum_{m=0}^{\infty} \phi\Big|_{k,L} T_{-}^{(t)}(m)(\tau,\mathfrak{z}) s^m,    
\end{equation}
which is a meromorphic modular form of weight $k$ on $\widetilde{\Orth}^+(U(t)\oplus U\oplus L)$. Here,
\begin{equation}
\phi\Big|_{k,L} T_{-}^{(t)}(0)(\tau,\mathfrak{z}) = c_\phi + \sum_{(n,\ell)>0} \sum_{\substack{(a,t)=1\\ a|(n,\ell)}} a^{k-1} \chi_t(a)\phi(0, \ell/a) q^{n}\zeta^\ell 
\end{equation}
where $c_\phi$ is a constant involving Bernoulli numbers that can be read off of \cite[Theorem 14.3]{Bor98}. In particular $\phi\Big|_{k,L} T_{-}^{(t)}(0)$ is a meromorphic Jacobi form of weight $k$ and index $L(0)$ on $\Gamma_0(t)$ with character $\chi_t$. \end{remark}

\section{Singular Borcherds products on lattices of prime level}\label{sec:prime}
In this section we classify all singular Borcherds products on lattices of prime level which are modular under the full orthogonal group. The important point is that we do not assume a priori that these products are reflective.  

\begin{theorem}\label{th:prime-level}
There are exactly $12$ lattices $M$ of prime level $p$ and not of type $\II_{26,2}(p)$ that admit a holomorphic Borcherds product of singular weight for $\Orth^+(M)$. For each $M$, the singular product is unique, and it is non-vanishing along only one 1-dimensional cusp of type $U(p)\oplus U\oplus L$. In Table \ref{tab: Sch} below we list these lattices and the inputs of the corresponding singular Borcherds products as scalar-valued weakly holomorphic modular forms on $\Gamma_0(p)$. 
\end{theorem}

\begin{table}[htp]
\caption{Holomorphic Borcherds products of singular weight on $\Orth^+(M)$ for lattices $M$ of prime level $p$}
\label{tab: Sch}
\renewcommand\arraystretch{1.5}
\[
\begin{array}{|c|c|c|c|c|c|}
\hline 
p & \text{genus}& M & f(\tau) & \text{zeros}  \\ 
\hline 
2 & \II_{18,2}(2_{\II}^{+10})& U(2)\oplus U\oplus \text{Barnes-Wall lattice} & \eta_{1^{-8}2^{-8}} & 1,2  \\
& \II_{10,2}(2_{\II}^{+2}) & U(2)\oplus U\oplus E_8 & 16\eta_{1^{-16}2^{8}} & 2  \\
& \II_{10,2}(2_{\II}^{+10}) & U(2)\oplus U\oplus E_8(2) & \eta_{1^{8}2^{-16}} & 1 \\
\hline 
3 & \II_{14,2}(3^{-8})& U(3)\oplus U\oplus \text{Coxeter--Todd lattice} & \eta_{1^{-6}3^{-6}}& 1,3  \\
 & \II_{8,2}(3^{-3})& U(3)\oplus U\oplus E_6 & 9\eta_{1^{-9}3^{3}} & 3  \\
 & \II_{8,2}(3^{-7})& U(3)\oplus U\oplus E_6'(3) & \eta_{1^{3}3^{-9}} & 1  \\
\hline  
5 & \II_{10,2}(5^{+6})& U(5)\oplus U\oplus \text{Maass lattice} & \eta_{1^{-4}5^{-4}} & 1,5  \\
& \II_{6,2}(5^{+3})& U(5)\oplus U\oplus A_4 & 5\eta_{1^{-5}5^{1}} & 5  \\
& \II_{6,2}(5^{+5})& U(5)\oplus U\oplus A_4'(5) & \eta_{1^{1}5^{-5}} & 1 \\
\hline 
7 & \II_{8,2}(7^{-5})& U(7)\oplus U\oplus \text{Barnes--Craig lattice} & \eta_{1^{-3}7^{-3}} & 1,7 \\
\hline 
11 & \II_{6,2}(11^{-4})& U(11)\oplus U\oplus \Lambda^{1^211^2} & \eta_{1^{-2}11^{-2}} & 1,11 \\
\hline 
23 & \II_{4,2}(23^{-3})& U(23)\oplus U\oplus \Lambda^{1^1 23^1} & \eta_{1^{-1}23^{-1}} & 1,23 \\
\hline
\end{array} 
\]
\end{table}

In Table \ref{tab: Sch}, $\Lambda^{1^1 23^1}$ and $\Lambda^{1^211^2}$ stand for the fixed-point sublattices of the Leech lattice associated with Conway's elements of cycle shapes $1^1 23^1$ and $1^2 11^2$, respectively (see the next section). The zeros of types $1$ and $p$ respectively denote the Heegner divisors $H(1)$ and $H(1/p)$ defined by
\begin{equation}\label{eq:Heegner}
H(1)=\sum_{\substack{\lambda\in M\\\lambda^2=2}} \lambda^\perp, \quad H(1/p)=\sum_{\substack{\lambda\in M'\\\lambda^2=2/p}} \lambda^\perp.    
\end{equation}
The singular Borcherds product on $M$ with divisor $H(1)$ is identified with the singular Borcherds product on $M'(p)$ with divisor $H(1/p)$. Therefore, there are only $9$ distinct singular Borcherds products in Theorem \ref{th:prime-level}.

\begin{remark}
The above 12 singular Borcherds products were systematically constructed by Scheithauer \cite{Sch04}. In \cite{Sch06} Scheithauer proved that they are the only reflective automorphic products of singular weight on lattices of prime level which have only simple zeros and are modular under the full orthogonal group. Scheithauer defined reflective automorphic products as reflective Borcherds products for which every nonzero Fourier coefficient in the principal part of the input corresponds to a nonempty quadratic divisor. This is why $\II_{26,2}(p)$ did not occur in his classification.  Scheithauer also showed that these singular Borcherds products correspond to the conjugacy classes of prime level of Conway's $\mathrm{Co}_0$ group and they coincide with the corresponding twisted denominator functions of the fake monster Lie algebra. 
\end{remark}

\begin{proof}
Let $M$ be an even lattice of signature $(l,2)$ and prime level $p$ with $l\geq 3$. Assume that $M$ is not isomorphic to the $p$-rescaling $\II_{26,2}(p)$ of $\II_{26,2}$. Suppose $F$ is a holomorphic Borcherds product of singular weight for $\Orth^+(M)$. By Lemma \ref{lem:1-cusp}, $F$ does not vanish at some 1-dimensional cusp denoted by $P_F$. By \cite[Lemma 5.1]{Bru14} and Lemma \ref{lem:U}, $P_F$ is represented by one of the decompositions $M=2U\oplus L$, $2U(p)\oplus L$ or $U(p)\oplus U\oplus L$. The first two possibilities can be ruled out:
\begin{enumerate}
\item By Theorem \ref{th:Leech-type} and Remark \ref{rk:prime-symmetric}, $P_F$ is not represented by a decomposition $M=2U\oplus L$.
\item If $P_F$ is represented by a decomposition $M=2U(p)\oplus L$, then we can view $F$ as a singular Borcherds product on $2U\oplus L'(p)$ which does not vanish at the corresponding 1-dimensional cusp, because $M'(p)$ is either unimodular or of level $p$ and
$$
\Orth^+(M) = \Orth^+(M') = \Orth^+(M'(p)).
$$
We see from (1) that $2U\oplus L'(p) = \II_{26,2}$, which yields  $M=\II_{26,2}(p)$. By \cite[Theorem 1.4]{Bru14}, one can construct Borcherds' form $\Phi_{12}$ as a Borcherds product on $\II_{26,2}(p)$ (see also \cite[Remark 3.10]{Wan23b}). Clearly, $\Phi_{12}$ is the unique reflective Borcherds product on $\II_{26,2}(p)$. 
\end{enumerate}

It follows that $F$ does not vanish identically at the 1-dimensional cusp determined by a splitting $M=U(p)\oplus U\oplus L$. By \cite[Corollary 5.5]{Sch15}, there exists a scalar-valued weakly holomorphic modular form $f(\tau)$ of weight $-\frac{1}{2}\rank(L)$ and character $\chi_{D_M}$ on $\Gamma_0(p)$ such that $F$ is the Borcherds product of $F_{\Gamma_0(p),f,0}$ (see Theorem \ref{th:scalar-to-vector}). 
By Lemma \ref{lem:special-input}, the input into $F$ as a sequence of (weakly holomorphic) Jacobi forms in the sense of Section \ref{sec:isomorphism} is $(\phi_1,\phi_p)$, where
\begin{align*}
\phi_1(\tau,\mathfrak{z})&=f(\tau)\Theta_{L,0}(\tau,\mathfrak{z}) \in J_{0,L}^!(\Gamma_0(p)),   \\
\phi_p(\tau,\mathfrak{z})&= \sum_{A\in \SL_2(\ZZ)/\Gamma_0(p)} \phi_1\big|_{0,L}A(\tau,\mathfrak{z}) \\
&=\sum_{n\in \ZZ}\sum_{\ell\in L'} \Big(\phi_1(n,\ell) + p\big(\phi_1\big|_{0,L}S\big)(n,\ell)\Big)q^n\zeta^\ell \in J_{0,L}^!(\SL_2(\ZZ)). 
\end{align*}
We have assumed that $F$ does not vanish at the 1-dimensional cusp related to $U(p)\oplus U\oplus L$. Therefore, the zeroth Fourier--Jacobi coefficient of $F$ on $U(p)\oplus U\oplus L$ is nonzero and it defines a scalar-valued holomorphic modular form of weight $\frac{1}{2}\rank(L)$ and some character on $\Gamma_0(p)$. Theorem \ref{th:FJ-level-N} implies that $F$ has the following properties:
\begin{itemize}
\item[1.] The Weyl vector of $F$ is of the form $(A,0,0)$.
\item[2.] The $q^0$-terms of $\phi_1$ and $\phi_p$ are constants, denoted respectively by $d_1$ and $d_p$. By Lemma \ref{lem:rationality-1}, $d_1, d_p \in \ZZ$. 
\item[3.] The zeroth Fourier--Jacobi coefficient of $F$ is the eta quotient
\begin{equation}
\theta(\tau) = \eta(\tau)^{d_1} \eta(p\tau)^{\frac{d_p-d_1}{p}}.      
\end{equation}
\item[4.] The first Fourier--Jacobi coefficient of $F$ is given by 
$$
-\theta(\tau)\phi_1(\tau,\mathfrak{z}) = - \theta(\tau)f(\tau)\Theta_{L,0}(\tau,\mathfrak{z}),
$$
and this has to be a holomorphic Jacobi form of singular weight and index $L$. Therefore, there exists a constant $c$ such that 
$$
    - \theta(\tau)f(\tau)\Theta_{L,0}(\tau,\mathfrak{z}) = c \Theta_{L,0}(\tau,\mathfrak{z}),
$$
that is,
\begin{equation}\label{eq:f}
 f(\tau) = -\frac{c}{\theta(\tau)}, \quad \phi_1(\tau,\mathfrak{z}) = - c\frac{ \Theta_{L,0}(\tau,\mathfrak{z})}{\theta(\tau)}.
\end{equation}
We see from the Fourier expansions of $f$ and $\theta$ at infinity that $A$ is a non-negative integer.
\end{itemize}

By construction, the Fourier expansion of $\theta(\tau)$ starts 
\begin{equation}
\theta(\tau) = q^A - d_1 q^{A+1} + O(q^{A+2}) = \sum_{n=0}^\infty a_n q^n.     
\end{equation}
The weight of $F$ equals the weight of $\theta$, so
\begin{equation}\label{eq:wtF}
    \rank(L) = d_1 + \frac{d_p - d_1}{p}. 
\end{equation}
In view of the $q$-order of $\theta(\tau)$, we have
\begin{equation}
    A = \frac{d_p}{24}.
\end{equation}

The Fourier expansion of $F$ is either symmetric or anti-symmetric upon swapping $\tau$ and $w$. We will argue by cases.

\textbf{I. The product $F$ does not vanish on  $H(1)$.} (see \eqref{eq:Heegner}) In this case, we have $F(\tau,\mathfrak{z},\omega) = F(\omega, \mathfrak{z},\tau)$ and therefore $c=a_1$. If $A>1$, then $a_1=0$ and thus $\phi_1=0$, which is impossible. If $A=1$, then $a_1=1$, $d_p=24$ and therefore $\phi_1=-q^{-1}-d_1+O(q)$, which yields $d_1=0$ because $[\phi_1]_{q^0}=d_1$. Then the multiplicity of $H(1)$ in the divisor of $F$ is $\frac{1}{p}[(p-1)\phi_1(-1,0)+\phi_p(-1,0)]=-1$, leading to a contradiction.

Therefore, $A=0$ and then $a_1=-d_1$. In this case, $d_p=0$ and
$$
\phi_1(\tau,\mathfrak{z}) = d_1\frac{ \Theta_{L,0}(\tau,\mathfrak{z})}{\theta(\tau)} = d_1 + O(q). 
$$

We classify the lattices by means of the explicit description of $F_{\Gamma_0(p),f,0}$ given in Remark \ref{rk:Scheithauer-lifting}.

The smallest $q$-order appearing in $f|S$ (denoted  $\delta=-\frac{1}{24}(d_1 - d_1/p^2)$) is negative and defines a zero divisor of $F$ whose multiplicity is one (see \cite[Theorem 1.2]{WW23}). Therefore, (noting that $|D_M|=p^2|D_L|$)
$$
\frac{d_1 p^{-\frac{d_1}{2p}}}{\sqrt{|D_L|}} =1,
$$
which implies that 
\begin{equation}
|D_L| = p^{2a-p^{a-1}}, \quad \text{where} \; d_1=p^a,
\end{equation}
where $a \ge 1$ is an integer. By \eqref{eq:wtF},
$$
\rank(L) = p^{a-1}(p - 1).
$$
The multiplicity of the divisor associated with the $q^{\delta+1/p}$-term of $f|S$ is $-p^{a-1}$, which is not equal to $1$. Therefore, this divisor is empty, which implies $\delta + \frac{1}{p} \geq 0$. Since $p\delta$ is integral, it follows that $\delta + \frac{1}{p}=0$, that is,
\begin{equation}\label{eq:2}
p^{a-1} + 24 = p^{a+1}. 
\end{equation}
The only solutions to \eqref{eq:2} are 
$$
(p,a)=(2,4), \quad (3,2), \quad (5,1).
$$

When $(p,a)=(2,4)$, $(\rank(L), |D_L|)=(8,1)$, so $L=E_8$ and $f(\tau)=16 \eta(2\tau)^8 / \eta(\tau)^{16}$. 

When $(p,a)=(3,2)$, $(\rank(L), |D_L|)=(6,3)$, so $L=E_6$ and $f(\tau)=9\eta(3\tau)^3 / \eta(\tau)^9$. 

When $(p,a)=(5,1)$, $(\rank(L), |D_L|)=(4,5)$, so $L=A_4$ and $f(\tau)=5\eta(5\tau)/\eta(\tau)^5$. 

Altogether, there are three symmetric singular Borcherds products on lattices of prime level, and the underlying lattices are $U(2)\oplus U\oplus E_8$, $U(3)\oplus U \oplus E_6$ and $U(5)\oplus U\oplus A_4$, respectively. For each of them, the Weyl vector is zero, the product $F$ does not vanish on any divisors $v^\perp$ with $v\in U\oplus L$, and $\phi_p$ is identically zero. 

\vspace{2mm}

\textbf{II. The product $F$ vanishes on $H(1)$.} In this case,  $F(\tau,\mathfrak{z},\omega) = -F(\omega, \mathfrak{z},\tau)$, and therefore $c=-a_1$. If $A>1$, then $a_1=0$ and thus $\phi_1=0$, which is impossible. If $A=0$, then $d_p=0$, $a_1=-d_1$ and thus $\phi_1=-d_1+O(q)$. We conclude from $[\phi_1]_{q^0}=d_1$ that $d_1=0$ and then $\theta=1$, a contradiction.

Therefore, $A=1$ and then $a_1=1$. In this case, $d_p=24$ and
$$
\phi_1(\tau,\mathfrak{z}) = \frac{ \Theta_{L,0}(\tau,\mathfrak{z})}{\theta(\tau)} = q^{-1} + d_1 + O(q). 
$$
The above $q^0$-term implies that 
$$
\Theta_{L,0}(\tau,\mathfrak{z}) = 1 + O(q^2),
$$
that is, $L$ has no 2-roots. 

The smallest $q$-order appearing in $f|S$ will be denoted $\delta=-\frac{1}{24}(d_1 + (24-d_1)/p^2)$. Since $1/f$ is holomorphic (see \eqref{eq:f}), we have $\delta\leq 0$. When $\delta=0$, the product $F$ vanishes only on the Heegner divisor $H(1)$, so $F$ can be viewed as a product on $M'(p)=U\oplus U(p) \oplus L'(p) $ which vanishes only on the Heegner divisor $H(1/p)$, and we can apply the classification in Case I. Indeed, $\delta=0$ implies that
$$
d_1 = - \frac{24}{p^2-1}. 
$$
The only integer solutions are 
$$
(p, d_1) = (2, -8),\; (3, -3), \; (5, -1),
$$
and the corresponding $f(\tau)$ are respectively
$$
\eta(\tau)^8 / \eta(2\tau)^{16}, \; \eta(\tau)^3 / \eta(3\tau)^9, \; \eta(\tau) / \eta(5\tau)^5. 
$$

When $\delta<0$, the $q^\delta$-term in $f|S$ defines a simple zero divisor of $F$. Therefore, 
$$
\frac{p^{\frac{24-d_1}{2p}}}{\sqrt{|D_L|}} =1,
$$
which implies that 
\begin{equation}
|D_L| = p^{\frac{24-d_1}{p}}.
\end{equation}
Therefore, $(24-d_1)/p$ is a non-negative integer denoted $a$. 
Equation \eqref{eq:wtF} implies
$$
\rank(L) = d_1 + a.
$$
If $a=0$, then $d_1=24$, $\rank(L)=24$ and $L$ is unimodular. In this case, $\delta=-1$. It follows that the multiplicity of $H(1)$ in the divisor of $F$ is 2, which is impossible. Therefore, $a\geq 1$.

From $|D_L|=p^a$ we find $\rank(L) \geq a$ and thus $d_1\geq 0$. It follows that 
\begin{equation}\label{eq:5}
    0 < pa \leq 24. 
\end{equation}

The multiplicity of the divisor corresponding to the $q^{\delta+1/p}$-term of $f|S$ is $a$. 

If $a>1$, then the corresponding divisor must be empty and thus $\delta + \frac{1}{p} \geq 0$. Since $p\delta$ is integral (equivalently $24|(p^2-1)a$), it follows that $\delta + \frac{1}{p} = 0$, i.e.
\begin{equation}\label{eq:IIa}
    a(p+1)=24, \quad a>1.
\end{equation}

If $a=1$, then the multiplicity of the divisor corresponding to the $q^{\delta+2/p}$-term of $f|S$ is at least 2, which yields 
\begin{equation}\label{eq:IIb}
\delta + \frac{2}{p} \geq 0, \quad \text{i.e.} \quad p + \frac{47}{p} \geq 24, \quad \text{i.e. $p=2$ or $p\geq 23$}. 
\end{equation}
If $a=1$ and $p=2$, $d_1=22$ and thus $\rank(L)=23$, which is impossible. The only solutions to \eqref{eq:5}, \eqref{eq:IIa} and \eqref{eq:IIb} are the following:
\begin{enumerate}
    \item $p=23$: $a=1$ and $d_1=1$. Thus $\rank(L)=2$, $|D_L|=23$ and $f(\tau)=\eta(\tau)^{-1}\eta(23\tau)^{-1}$;
    \item $p=11$: $a=2$ and $d_1=2$. Thus $\rank(L)=4$, $|D_L|=11^2$ and $f(\tau)=\eta(\tau)^{-2}\eta(11\tau)^{-2}$;
    \item $p=7$: $a=3$ and $d_1=3$. Thus $\rank(L)=6$, $|D_L|=7^3$ and $f(\tau)=\eta(\tau)^{-3}\eta(7\tau)^{-3}$;
    \item $p=5$: $a=4$ and $d_1=4$. Thus $\rank(L)=8$, $|D_L|=5^4$ and $f(\tau)=\eta(\tau)^{-4}\eta(5\tau)^{-4}$;
    \item $p=3$: $a=6$ and $d_1=6$. Thus $\rank(L)=12$, $|D_L|=3^6$ and $f(\tau)=\eta(\tau)^{-6}\eta(3\tau)^{-6}$;
    \item $p=2$: $a=8$ and $d_1=8$. Thus $\rank(L)=16$, $|D_L|=2^8$ and $f(\tau)=\eta(\tau)^{-8}\eta(2\tau)^{-8}$.
\end{enumerate}
The property that $L$ has no $2$-roots uniquely determine the above lattices $L$ in their genera. 
\end{proof}

We remark that \cite[Theorem 1.5]{WW23} may simplify the proof of Theorem \ref{th:prime-level}.

\section{Moonshine for the Conway group}\label{sec:moonshine}
In this section, we prove the moonshine conjecture for the Conway group which was proposed by Borcherds and Scheithauer. More precisely, we will show that all twisted denominator identities of the fake monster algebra can be realized as Borcherds products of  appropriate weights. 

\subsection{The Leech lattice} 
In this section we review some basic properties of the Leech lattice, its automorphism group, and the associated eta quotients.

Let $\Lambda$ be the Leech lattice: the unique positive-definite even unimodular lattice of rank $24$ without roots \cite{CS99}. The integral orthogonal group of $\Lambda$ is the Conway group $\mathrm{Co}_0$. Each $g \in \mathrm{Co}_0$ acts on $\Lambda \otimes \mathbb{Q}$ linearly with a characteristic polynomial of the form 
$$
\mathrm{det}(I - gX) = \prod_{k=1}^{\infty} (1 - X^k)^{b_k},\quad b_k\in\ZZ.
$$ 
The symbol $\prod_{b_k \neq 0} k^{b_k}$ is called the \emph{cycle shape} of $g$ and it depends only on the $\mathrm{Co}_0$-conjugacy class. For any positive integer $m$, the cycle shape of $g^m$ is then
\begin{equation}
\prod_{k=1}^\infty \Big( k/(k,m) \Big)^{(k,m)b_k}.    
\end{equation}
In particular, the order of $g$ (denoted $n_g$) equals the least common multiple of all $k$ with $b_k\neq 0$, and the trace of $g^m$ on $\Lambda\otimes\QQ$ is
\begin{equation}
\mathrm{tr}(g^m) = \sum_{k|m} kb_k.  
\end{equation}
By M\"{o}bius inversion,
\begin{equation}\label{eq:b_m}
b_m =  \frac{1}{m}  \sum_{k|m} \mu\left( \frac{m}{k} \right) \mathrm{tr}(g^k). 
\end{equation}
Note that $\sum_{k|n_g} kb_k = 24$ and that $\sum_{k|n_g} b_k$ is always even. 

The eta quotient associated to $g$ is 
\begin{equation}
\eta_g(\tau) := \prod_{k|n_g} \eta(k \tau)^{b_k} = q^{\frac{1}{24}\sum_{k|n_g} kb_k} \prod_{k|n_g }\prod_{n=1}^\infty (1 - q^{kn})^{b_k}.    
\end{equation}
The function $\eta_g$ is a weakly holomorphic modular form of weight $\frac{1}{2}\sum_{k|n_g}b_k$ and some character $\chi_g$ on $\Gamma_0(N_g)$, where $N_g$ is the smallest multiple of $n_g$ for which $N_g \cdot \sum_{k|n_g} b_k/k$ is divisible by $24$. The number $N_g$ is called the \textit{level} of $g$. Since the Fourier expansion of $\eta_g$ involves only integers, $\chi_g$ acts trivially on $\Gamma_1(N_g)$. 

Let $\Lambda^g$ be the fixed-point sublattice (see \cite{HL90} for its description)
$$
\Lambda^g = \{v \in \Lambda: \; g(v) = v\}.
$$ 
Then $\Lambda^g$ has (even) rank $\sum_{k|n_g}b_k$. Clearly, $\Lambda^g=0$ if and only if $\sum_{k|n_g}b_k=0$. Let $l_g$ denote the level of $\Lambda^g$; if $\Lambda^g= 0$ then we set $l_g=1$. Note that $l_g|N_g$ and that the order of each coset in $(\Lambda^g)'/\Lambda^g$ divides $n_g$.

The Jacobi theta function 
\begin{equation}
\Theta_g(\tau, \mathfrak{z}) = \sum_{\ell \in \Lambda^g} q^{\ell^2/2} \zeta^{\ell}=1+O(q^2), \quad \mathfrak{z}\in \Lambda^g\otimes\CC, \; \zeta^\ell = e^{2\pi i(\ell,\mathfrak{z})} 
\end{equation}
is a holomorphic Jacobi form of singular weight $\frac{1}{2}\rank(\Lambda^g)$ and index $\Lambda^g$ with the same character $\chi_g$ on $\Gamma_0(l_g)$. When $\Lambda^g=0$, we define $\Theta_g(\tau,\mathfrak{z})=1$.  (Recall that weakly holomorphic Jacobi forms of weight $k$ and index $0$ on $\Gamma_0(N)$ are the same as weakly holomorphic scalar-valued modular forms of weight $k$ on $\Gamma_0(N)$.) Therefore:

\begin{lemma}\label{lem:Jacobi-untwisted}
For any $g\in \mathrm{Co}_0$,  the function 
$$
\Theta_g(\tau,\mathfrak{z}) / \eta_g(\tau) = q^{-1} + \mathrm{tr}(g) + O(q) 
$$
defines a weakly holomorphic Jacobi form of weight $0$ and index $\Lambda^g$ on $\Gamma_0(N_g)$ with trivial character. 
\end{lemma}

$\mathrm{Co}_0$ has $167$ conjugacy classes with $160$ distinct cycle shapes; the cycle shapes 
$2^1 22^1$,
$1^1 23^1$,
$1^{-1} 2^1 23^{-1} 46^1$, 
$4^1 8^{-1} 28^{-1} 56^1$,
$1^1 3^{-1} 13^{-1} 39^1$,
$1^{-1} 2^1 3^1 6^{-1} 13^1 26^{-1} 39^{-1} 78^1$,
$2^1 8^{-1} 10^{-1} 40^1$ are realized by two conjugacy classes. Only the first two yield non-trivial fixed-point sublattices, and the sublattices coming from distinct conjugacy classes with the same cycle shape turn out to be isomorphic. Therefore, there are exactly $95$ conjugacy classes (with $90$ distinct cycle shapes) having trivial fixed-point sublattices. 

\subsection{The central extension of the Leech lattice}
There is a unique central extension of the Leech lattice,
$$
0 \to \{\pm 1\} \to \hat{\Lambda} \to \Lambda \to 0
$$
in which the commutator of the preimages of $u,v\in \Lambda$ is $(-1)^{(u,v)}$. This central extension plays an important role in the twisted denominator identities of the fake monster algebra. We describe this extension and review some useful properties following \cite[\S 7]{EMS20}.

Let $\varepsilon : \Lambda \times \Lambda \to \{\pm 1\}$ be a $2$-cocycle satisfying
$$
\varepsilon(v, v) = (-1)^{(v,v)/2} \quad \text{and} \quad \frac{\varepsilon(u,v)}{\varepsilon(v,u)} = (-1)^{(u,v)}
$$
for all $u,v\in \Lambda$. We identify $\Lambda$ with the set $\{ e^v: v\in \Lambda \}$ with multiplication $e^u e^v=e^{u+v}$. The extension $\hat{\Lambda}$ can then be defined as the set $\{ \pm e^v: v\in \Lambda \}$ equipped with the twisted multiplication
$$
e^u e^v = \varepsilon(u,v)e^{u+v}.
$$
For any $g\in \mathrm{Co}_0$, there exists a function $\lambda_g: \Lambda \to \{\pm 1\}$ such that 
$$
\frac{\varepsilon(u, v)}{\varepsilon(g(u), g(v))} = \frac{\lambda_g(u)\lambda_g(v)}{\lambda_g(u+v)}, \quad u,v\in \Lambda.
$$
$g$ can then be lifted to an automorphism $\hat{g}$ of $\hat{\Lambda}$ via 
$$
\hat{g}(e^v) = \lambda_g(v)e^{g(v)}
$$
so that $\hat{g}(e^u e^v)= \hat{g}(e^u) \hat{g}(e^v)$ for all $u,v\in \Lambda$.  In fact, the above $\lambda_g$ can always be chosen such that $\lambda_g(v)=1$ for all $v\in \Lambda^g$, in which case we call the induced $\hat{g}$ a \textit{standard lift} of $g$. By \cite[Proposition 7.1]{EMS20}, the standard lift is unique up to conjugation in $\mathrm{Aut}(\hat{\Lambda})$. 

From now on, $\hat{g}$ always denotes a standard lift of $g$. For any positive integer $d$, $(\hat{g})^d$ is a lift of $g^d$, but it is not necessarily a standard lift. We will say that $g$ has a \textit{nice lift} if $(\hat{g})^d$ is a standard lift of $g^d$ for all $d|n_g$. 

For any  $v\in \Lambda^{g^d}$ we define
\begin{equation}\label{eq:xi}
 \xi_d(v)=
\begin{cases}
1, & \text{if $n_g$ or $d$ is odd;}\\
(-1)^{(v,g^{d/2}(v))}, &\text{if both $n_g$ and $d$ are even.}   
\end{cases}
\end{equation}
By \cite[Propositions 7.3, 7.4]{EMS20}, for any $v\in \Lambda^{g^d}$ we have 
\begin{equation}\label{eq:lift-action}
   (\hat{g})^d(e^v) = \xi_d(v) e^v.
\end{equation}
In particular:
\begin{enumerate}
    \item If $n_g$ is odd then $(\hat{g})^d$ is always a standard lift of $g^d$, and therefore $g$ has a nice lift.
    \item If $\Lambda^g$ is trivial, then $g$ has a nice lift. This follows from the fact that when $d$ is even $v_d:=\sum_{t=0}^{d-1}g^t(v) \in \Lambda^g$ and $(v, g^{d/2}(v))\equiv (v,v_d)\mod 2$.
    \item Let $\hat{n}_g$ denote the order of $\hat{g}$. If $n_g$ is even and  $(v,g^{n_g/2}(v))$ is odd for some $v\in \Lambda$ then $\hat{n}_g= 2n_g$; otherwise $\hat{n}_g=n_g$. Note that $\hat{n}_g|N_g$.
\end{enumerate}

\begin{remark}\label{rk:xi=1}
In some common cases $\xi_d \equiv 1$ holds automatically. Suppose that both $n_g$ and $d$ are even and $\Lambda^{g^d}$ is non-trivial and define $$\Lambda^{g^{d/2},2} = (\Lambda^{g^{d/2}})^\perp\cap \Lambda^{g^d}.$$
Any $v \in \Lambda^{g^d}$ can be uniquely written as $v=v_1 + v_2$ where $v_1 \in (\Lambda^{g^{d/2}})'$ and $v_2\in (\Lambda^{g^{d/2},2})'$.  Since $g^{d/2}(v) +v\in \Lambda^{g^{d/2}}$ and $g^{d/2}(v_1)=v_1$,  we have $g^{d/2}(v_2)=-v_2$ and $v+g^{d/2}(v)=2v_1$.  Therefore, we have
\begin{enumerate}
    \item $2v_1 \in \Lambda^{g^{d/2}}$;
    \item $(v, g^{d/2}(v))\in 2\ZZ$ if and only if $(v_1,v_1)\in \ZZ$.
\end{enumerate}
In particular, if $\Lambda^{g^{d/2}}$ has level $m$ or $2m$ for some odd integer $m$, then $(v_1,v_1)$ is always integral and thus $\xi_d=1$ on $\Lambda^{g^d}$. 
\end{remark}

It is not hard to prove the following results by direct calculation. 

\begin{lemma}
The standard lift of $g$ has order $2n_g$ if and only if the cycle shape of $g$ is $2^{12}$, $4^6$, $2^3 6^3$, $6^4$, $2^2 10^2$, $12^2$, $4^1 20^1$ or $2^1 22^1$.  
\end{lemma}
\begin{proof}
This is because $\hat{g}$ has order $2n_g$ if and only if $n_g$ is even and $g^{n_g/2}$ has cycle shape $2^{12}$. (All other elements $h \in \mathrm{Co}_0$ of order $2$ satisfy $(v, h(v))\in 2\ZZ$ for all $v\in \Lambda$). 
\end{proof}

\begin{lemma}\label{lem:not-nice}
The element $g \in \mathrm{Co}_0$ fails to have a nice lift if and only if $g$ has one of the following $20$ cycle shapes:
\begin{align*}
&2^{12} (d=2),\quad 
1^4 2^2 4^4 (d=2),\quad
1^{-4} 2^6 4^4 (d=2),\quad
4^6 (d=4),\quad
2^3 6^3 (d=2,6),\\ 
&6^4 (d=2,6),\quad
1^2 2^1 4^1 8^2 (d=2,4),\quad
1^{-2} 2^3 4^1 8^2 (d=2,4),\quad
2^2 10^2 (d=2, 10),\\
&1^1 2^2 3^1 4^{-2} 12^2 (d=2, 6),\quad
1^{-1} 2^3 3^{-1} 4^{-2} 6^1 12^2 (d=2, 6),\quad
1^{-2} 2^2 3^2 4^1 12^1 (d=2, 6),\\
&1^2 3^{-2} 4^1 6^2 12^1 (d=2, 6),\quad
12^2 (d=4, 12),\quad
4^1 20^1 (d=4, 20),\quad
1^{-1} 2^2 4^{-1} 5^1 20^1 (d=2, 10),\\
&1^1 2^1 4^{-1} 5^{-1} 10^1 20^1 (d=2, 10),\quad
2^1 22^1 (d=2, 22),\\
&1^{-1} 2^1 3^1 4^1 8^{-1} 24^1 (d=2, 4, 6, 12),\quad
1^1 3^{-1} 4^1 6^1 8^{-1} 24^1 (d=2, 4, 6, 12).
\end{align*}
Next to each cycle shape we list all $d|n_g$ with $\xi_d\neq 1$; there are two conjugacy classes of cycle shape $2^1 22^1$ and the corresponding $\xi_d$ are the same, so this entry appears only once. In particular,  every $g$ of square-free level has a nice lift. \end{lemma}

\subsection{The twisted denominator identity of the fake monster algebra}

In 1990 Borcherds \cite{Bor90} constructed the fake monster algebra and determined its denominator identity. The fake monster algebra $G$ is a generalized Kac--Moody algebra whose root lattice is $\mathrm{II}_{25,1}=U\oplus \Lambda$. 

We write vectors $v\in \II_{25,1}$ as $(n,r,m)\in \ZZ \times \Lambda \times \ZZ$ with norm $r^2-2nm$. In this coordinate system, the Weyl vector of $G$ is $\rho=(1,0,0)$. The real simple roots of $G$ are the vectors $\alpha \in \mathrm{II}_{25,1}$ satisfying $\alpha^2=2$ and $(\rho,\alpha)=-1$. The imaginary simple roots are $n\rho$ with multiplicity $24$ for all positive integers $n$. The Weyl group $W$ is generated by reflections associated to real simple roots. Borcherds' denominator identity is
\begin{equation}
e^\rho\prod_{\alpha\in \mathrm{II}_{25,1}^+} (1-e^\alpha)^{c(-\alpha^2/2)} = \sum_{w\in W} \det(w) w\left(e^\rho\prod_{n=1}^\infty\Big(1-e^{n\rho}\Big)^{24}\right),  
\end{equation}
where the multiplicities of positive roots are determined by 
$$
\Delta^{-1}(\tau)=\sum_{n\in \ZZ} c(n)q^n = q^{-1} + 24 + O(q).
$$
This is the Fourier expansion of the Borcherds form $\Phi_{12}$ when $e^{\alpha}$ is interpreted as $e^{2\pi i (\alpha, Z )}$ (cf. \cite{Bor95}). 

This was proved by Borcherds as a cohomological identity. The group $\Orth(\hat{\Lambda})$ of automorphisms preserving the inner product acts naturally on $G$ by realizing $G$ as some cohomology of the vertex operator algebra of $\hat{\Lambda}$. By taking the trace of $\hat{g}$ over the cohomological identity, Borcherds obtained the twisted denominator identities (see \cite[\S 13]{Bor92} for $g$ having a nice lift and \cite[\S 5]{Sch11} for arbitrary $g$)
\begin{equation}\label{eq:twisted-denom}
    e^\rho \prod_{\alpha \in L_g^+}\Big(1-e^\alpha\Big)^{\mathrm{mult}(\alpha)} = \sum_{w\in W_g} \det(w) w\left( e^\rho \prod_{k|n_g} \prod_{n=1}^\infty \Big(1-e^{kn\rho}\Big)^{b_k} \right), 
\end{equation}
where $L_g=U\oplus \Lambda^g$, $\rho=(1,0,0)$, $W_g$ is the subgroup of $W$ of elements that map $L_g$ into $L_g$, and the multiplicities $\mathrm{mult}(\alpha)$ are given by
\begin{equation}
 \mathrm{mult}(\alpha) =  \sum_{dk|(\alpha,\, \hat{n}_g)}  \frac{\mu(k)}{dk} \mathrm{Tr}\Big((\hat{g})^d\big| \widetilde{E}_{\alpha/dk}\Big). 
\end{equation}
Here, $a|(\alpha, \hat{n}_g)$ means that $\alpha/a \in L_g'$ and $a|\hat{n}_g$; and if $E$ denotes the subalgebra of $G$ corresponding to the positive roots, then for $\beta=(n,r',m)\in L_g'$ the spaces $\widetilde{E}_{\beta}$ above are
$$
\widetilde{E}_{\beta} = \bigoplus_{\pi (r) = r'} E_{(n,r,m)},
$$
where $\pi: \Lambda\otimes\QQ \to \Lambda^g\otimes\QQ$ is the orthogonal projection.

Borcherds proved that the twisted denominator identity is the untwisted denominator identity of a generalized Kac--Moody superalgebra whose real simple roots are the roots $\alpha$ of $L_g$ satisfying $(\rho, \alpha) = -\alpha^2/2$ and imaginary simple roots are $m\rho$ with multiplicity $\sum_{k|(m,n_g)}b_k$ for all positive integers $m$. In \cite[\S 15, Example 3]{Bor95} Borcherds conjectured that the twisted denominator function is always an automorphic form of singular weight for some orthogonal group. 

Scheithauer formulated Borcherds' conjecture precisely and called it the moonshine conjecture for Conway's group (see \cite[\S 8]{Sch04}, \cite[\S 10]{Sch06}, \cite[\S 5]{Sch11}).  This predicts that the twisted denominator identity corresponding to $g$ defines a Borcherds product of weight $\frac{1}{2}\rank(\Lambda^g)$ on some orthogonal group of signature $(\rank(\Lambda^g)+2,2)$. Scheithauer proved this conjecture for all elements of $\mathrm{Co}_0$ of square-free level in \cite{Sch01, Sch04, Sch08} (including $68$ distinct cycle shapes), for elements of cycle shapes $1^4 2^2 4^4$ and $1^2 2^1 4^1 8^2$ in \cite{Sch09}, for elements of cycle shapes $2^{12}$, $3^8$ and $3^1 21^1$ in \cite{Sch15}, and for elements of cycle shapes $4^8 2^{-4}$, $1^3 3^{-2} 9^3$ and $3^{-1} 9^3$ in \cite{Sch17}. The main idea of the proof of Scheithauer is as follows. Suppose $g\in \mathrm{Co}_0$ has square-free level and $\Lambda^g$ is non-trivial (then $N_g$=$n_g$). Scheithauer proved that the lifting $F_{\Gamma_0(n_g), 1/\eta_g, 0}$ (see Theorem \ref{th:scalar-to-vector}) yields a holomorphic Borcherds product of singular weight on $\Orth^+(U(n_g)\oplus U\oplus \Lambda^g)$ whose Fourier expansion at the $0$-dimensional cusp related to $U(n_g)$ gives the twisted denominator identity corresponding to $g$. Moreover, he proved in \cite{Sch06} that these singular Borcherds products are reflective. The eight non-squarefree cycle shapes above were dealt with by a generalization of the lifting $F_{\Gamma_0(N), f, 0}$.  Scheithauer constructed the singular Borcherds products corresponding to the last five cycle shapes on lattices not of the form $U(n_g)\oplus U\oplus \Lambda^g$.  Following Scheithauer's argument, the conjecture was further proved for elements of cycle shapes $1^{-2} 2^3 4^1 8^2$, $1^{-2} 2^2 3^2 4^1 12^1$ and $1^{-4} 2^6 4^4$ in \cite{DW21}.

\subsection{Moonshine for Conway's group}
We give a uniform proof of the moonshine conjecture for Conway's group. Our proof is different from Scheithauer's approach and is based on the representation of Borcherds products on $U(N)\oplus U\oplus L$ in terms of Jacobi forms in \S \ref{sec:isomorphism}. The input forms into the Borcherds lift turn out to have natural expressions as weakly holomorphic Jacobi forms.

\begin{theorem}\label{th:moonshine}
Let $g$ be an element of $\mathrm{Co}_0$ of level $N_g$. Let $\hat{g}$ be a standard lift of $g$ of order $\hat{n}_g$. For any $d|N_g$ define
\begin{equation}
\phi_{g,d}(\tau,\mathfrak{z}) = \sum_{n\in \ZZ}\sum_{\ell \in (\Lambda^g)'} \mathrm{Tr}\Big( (\hat{g})^d\big| \widetilde{E}_{(n,\ell,1)}\Big) q^n \zeta^\ell, \quad (\tau,\mathfrak{z}) \in \HH \times (\Lambda^g\otimes\CC).   \end{equation}
Then we have the following:
\begin{enumerate}
    \item For every $d|N_g$, $\phi_{g,d}$ is a weakly holomorphic Jacobi form of weight $0$ and index $\Lambda^g$ on $\Gamma_0(N_g/d)$ with trivial character.
    
    \item The Borcherds lift of $(\phi_{g,d})_{d|N_g}$ is a modular form of weight $\frac{1}{2}\rank(\Lambda^g)$ and some character for $\widetilde{\Orth}^+(U(N_g)\oplus U \oplus \Lambda^g)$. The Fourier expansion of this Borcherds product at the $1$-dimensional cusp related to $U(N_g)\oplus U$ is the twisted denominator corresponding to $g$. In particular, when $\Lambda^g$ is non-trivial, $\Borch((\phi_{g,d})_{d|N_g})$ is a holomorphic Borcherds product of singular weight on $U(N_g)\oplus U\oplus \Lambda^g$. 
\end{enumerate}
\end{theorem}

To prove part (1) we relate $\phi_{g, d}$ to the theta function of the fixed-point lattice of $g^d$. First we compute the value of $\mathrm{Tr}((\hat{g})^d\big|\widetilde{E}_\alpha)$ for $\alpha \in L_g'$. 

We recall some notation from \cite[\S 8]{Sch04}. For any positive integer $d$, $\Lambda^g$ is a primitive sublattice of $\Lambda^{g^d}$. We define $$\Lambda^{g,d}=(\Lambda^g)^\perp \cap \Lambda^{g^d}.$$ There exist a subgroup $G_{d,\Lambda^g}$ of $(\Lambda^g)'/\Lambda^g$, a subgroup $G_{d, \Lambda^{g,d}}$ of $(\Lambda^{g,d})'/\Lambda^{g,d}$ and an isomorphism $\gamma_d : G_{d, \Lambda^{g}} \to G_{d, \Lambda^{g,d}}$ such that every element of $\Lambda^{g^d} / (\Lambda^g\oplus \Lambda^{g,d})$ can be expressed as $r' + \gamma_d(r')$ for some class $r' \in G_{d, \Lambda^{g}}$. For $r'\in (\Lambda^g)'$, we define a theta function $\theta_{\xi_d, \gamma_d(r')}$ by 
\begin{equation}
\theta_{\xi_d, \gamma_d(r')}(\tau) = \sum_{v\in \gamma_d(r')+\Lambda^{g,d}} \xi_d(r'+v) q^{v^2/2}
\end{equation}
when $r' \in G_{d,\Lambda^g}$, where $\xi_d$ is defined in \eqref{eq:xi}, and $\theta_{\xi_d, \gamma_d(r')} = 0$ if $r' \not\in G_{d, \Lambda^g}.$ Note that $\theta_{\xi_d, \gamma_d(r')}$ depends only on the class of $r'$ modulo $\Lambda^g$.

The following lemma is an extension of \cite[Proposition 8.3]{Sch04}.
\begin{lemma}\label{lem:value-E}
For any positive integer $d$ and $\alpha=(n,r',m)\in L_g'$, the value of $\mathrm{Tr}\big((\hat{g})^d\big|\widetilde{E}_\alpha\big)$ is the coefficient of $q^{-\alpha^2/2}$ in 
$$
\theta_{\xi_d, \gamma_d(r')}(\tau) / \eta_{g^d}(\tau).
$$
In particular,  we have the following:
\begin{enumerate}
    \item $\mathrm{Tr}\big((\hat{g})^d\big|\widetilde{E}_\alpha\big)=0$ if $dr' \not\in \Lambda^g$. 
    \item For given $g$ and $d$, the value of $\mathrm{Tr}\big((\hat{g})^d\big|\widetilde{E}_\alpha\big)$ depends only on the norm of $\alpha$ and the coset of $\alpha$ modulo $L_g$. 
\end{enumerate}
\end{lemma}
\begin{proof}
Scheithauer \cite[Proposition 8.3]{Sch04} proved this lemma for elements $g \in \mathrm{Co}_0$ which have a nice lift (in which case $\xi_d=1$ always holds), but his proof can be extended to the general case. The no-ghost theorem implies that for arbitrary $g\in \mathrm{Co}_0$ the identity
\begin{equation*}
\mathrm{Tr}\Big((\hat{g})^d\big|\widetilde{E}_\alpha\Big) = \sum_{s\in S} \mathrm{Tr}\Big((\hat{g})^d\big| V_{1+nm}(r'+s) \Big)
\end{equation*}
holds, where $V_{1+nm}$ is the $L_0$-eigenspace of degree $1+nm$ of the vertex operator algebra of the Leech lattice and $S=\gamma_d(r') + \Lambda^{g,d}$ if $r'\in G_{d,\Lambda^g}$ and $S$ is empty otherwise. The space $V_{1+nm}(r'+s)$ is generated by bosonic oscillators and $e^{r'+s}$. By definition, $(\hat{g})^d(e^{r'+s})=\xi_d(r'+s)e^{r'+s}$, so the trace $\mathrm{Tr}\big((\hat{g})^d\big|\widetilde{E}_\alpha\big)$ can be computed as in Scheithauer's proof. We know from \cite[Proposition 8.2]{Sch04} that if $r'\in G_{d, \Lambda^g}$ then $dr' \in \Lambda^g$, proving claim (1); and claim (2) is clear.  
\end{proof}

\begin{remark}\label{rk:match}
The above lemma shows that for any positive integer $d$,
$$
\mathrm{Tr}\big((\hat{g})^d\big|\widetilde{E}_\alpha\big) = \mathrm{Tr}\big((\hat{g})^{(d,\hat{n}_g)}\big|\widetilde{E}_\alpha\big).
$$
Let $d_1=(d, \hat{n}_g)$ and $d_0=d/d_1$. We find that $g^{d_1}$ and $g^{d}$ have the same cycle shape and $\Lambda^{g^{d_1}}=\Lambda^{g^d}$. To show that $\xi_d=\xi_{d_1}$ on $\Lambda^{g^{d_1}}=\Lambda^{g^d}$, we only need to consider the case that both $n_g$ and $d$ are even. When $\hat{n}_g=2n_g$, since $d_0$ is coprime with the order of $g^{d_1/2}$, we see that $g^{d_1/2}$ and $g^{d/2}$ have the same cycle shape and that $\Lambda^{g^{d_1/2}}=\Lambda^{g^d/2}$, which yields $\xi_d=\xi_{d_1}$ by Remark \ref{rk:xi=1}. When $\hat{n}_g=n_g$ and $d_0=2a+1$ is odd, for any $v\in \Lambda^{g^{d_1}}$ we have 
$$
(v, g^{d/2}(v)) = (v, g^{d_1/2}g^{ad_1}(v)) = (v, g^{d_1/2}(v)),
$$
which yields $\xi_d=\xi_{d_1}$. Finally, if $d_0$ is even then $\xi_{d}=1$. We can write $n_g=d_1(2a+1)$ with $a \in\mathbb{N}_0$ because $(n_g/d_1, d_0)=1$. For any $v\in \Lambda^{g^{d_1}}$ we have 
$$
2\ZZ \ni (v, g^{n_g/2}(v)) = (v, g^{d_1/2}g^{ad_1}(v)) = (v, g^{d_1/2}(v)).
$$
Thus $\xi_{d_1}=1$ in this case as well. This discussion also shows that if $n_g=\hat{n}_g$ and $n_g/d$ is odd then $\xi_d=1$ on $\Lambda^{g^d}$. 
\end{remark}

\begin{proposition}\label{prop:Jacobi-input}
For any $d|N_g$, the function $\phi_{g,d}$ defined in Theorem \ref{th:moonshine} has the expression
$$
\phi_{g,d}(\tau,\mathfrak{z}) = \frac{\Theta_{\xi_d,g^d}\Big|_{\Lambda^g}(\tau,\mathfrak{z})}{\eta_{g^d}(\tau)},
$$
where $\Theta_{\xi_d,g^d}\Big|_{\Lambda^g}$ is the pullback of $\Theta_{\xi_d, g^d}$ along the natural embedding $\Lambda^g < \Lambda^{g^d}$ and 
$$
\Theta_{\xi_d,g^d}(\tau,z) = \sum_{v \in \Lambda^{g^d}} \xi_d(v) e^{\pi i(v,v)\tau + 2\pi i(v, z)}, \quad z\in \Lambda^{g^d}\otimes\CC.
$$
\end{proposition}
\begin{proof}
We calculate
\begin{align*}
\Theta_{\xi_d,g^d}\Big|_{\Lambda^g}(\tau,\mathfrak{z}) &= \sum_{v \in \Lambda^{g^d}} \xi_d(v) e^{\pi i(v,v)\tau + 2\pi i(v, \mathfrak{z})} \\
&= \sum_{r'\in G_{d,\Lambda^g} + \Lambda^g} \sum_{s\in \gamma_d(r') + \Lambda^{g,d}} \xi_d(r'+s) q^{r'^2/2+s^2/2} \zeta^{r'}  \\
&= \sum_{r'\in G_{d,\Lambda^g} + \Lambda^g} q^{r'^2/2}\zeta^{r'} \sum_{s\in \gamma_d(r') + \Lambda^{g,d}} \xi_d(r'+s) q^{s^2/2}. 
\end{align*}
By Lemma \ref{lem:value-E} we have
\begin{align*}
\frac{\Theta_{\xi_d,g^d}\Big|_{\Lambda^g}(\tau,\mathfrak{z})}{\eta_{g^d}(\tau)} & = \sum_{r'\in G_{d,\Lambda^g} + \Lambda^g} q^{r'^2/2}\zeta^{r'} \cdot \frac{\theta_{\xi_d,\gamma_d(r')}(\tau)}{\eta_{g^d}(\tau)} \\
& = \sum_{r'\in G_{d,\Lambda^g} + \Lambda^g} q^{r'^2/2}\zeta^{r'} \sum_{n \in \ZZ } \mathrm{Tr}\Big((\hat{g})^d\big|\widetilde{E}_{(n,r',1)}\Big) q^{n-r'^2/2} \\ 
&= \sum_{n\in \ZZ} \sum_{r' \in (\Lambda^g)'} \mathrm{Tr}\Big((\hat{g})^d\big|\widetilde{E}_{(n,r',1)}\Big) q^n \zeta^{r'}. \qedhere
\end{align*}
\end{proof}
Note that the zero-value $\phi_{g,d}(\tau,0)$ is the trace of $(\hat{g})^d$ acting on the Leech lattice vertex operator algebra (see \cite[Proposition 7.5]{EMS20}), i.e.
$$
\phi_{g,d}(\tau,0) = \mathrm{tr}_V (\hat{g})^d q^{L_0-1} =\eta_{g^d}(\tau)^{-1}\sum_{v\in \Lambda^{g^d}} \xi_d(v) q^{v^2/2}.
$$

Using Proposition \ref{prop:Jacobi-input} we can read off the behavior of $\phi_{g,d}$ under $\mathrm{SL}_2(\mathbb{Z})$. Suppose that both $n_g$ and $d$ are even, $\Lambda^{g^d}$ is non-trivial and $\xi_d\neq 1$ on $\Lambda^{g^d}$.  We define 
\begin{align}
    \Lambda^{g^d,+} &= \{ v \in \Lambda^{g^d} : \xi_d(v) =1 \},\\
    \Lambda^{g^d,-} &= \{ v \in \Lambda^{g^d} : \xi_d(v) =-1 \}.
\end{align}
Then $\Lambda^{g^d,+}$ is a sublattice of index $2$ of $\Lambda^{g^d}$, and $\Lambda^{g^d,-}$ is its non-trivial coset: $\Lambda^{g^d,-}=\beta + \Lambda^{g^d,+}$ for any $\beta \in \Lambda^{g^d}$ satisfying that $(\beta, g^{d/2}(\beta))$ is odd. Clearly, $2\beta \in \Lambda^{g^d,+}$. We have the decomposition (with notation as in Example \ref{ex:Theta})
\begin{equation}
\Theta_{\xi_d,g^d}(\tau,z) = \Theta_{\Lambda^{g^d,+},0}(\tau,z) - \Theta_{\Lambda^{g^d,+},\beta}(\tau,z).     
\end{equation}
Let $l_{g^d}^+$ be the level of $\Lambda^{{g^d},+}$. Then $l_{g^d}^+$ is even and $l_{g^d} | l_{g^d}^+$ (where $l_{g^d}$ is the level of $\Lambda^{g^d}$). We see from Example \ref{ex:Theta} and \cite[Proposition 4.5]{Sch09} that $\Theta_{\Lambda^{g^d,+},0}$ and $\Theta_{\Lambda^{g^d,+},\beta}$ are holomorphic Jacobi forms of weight $\frac{1}{2}\rank(\Lambda^{g^d})$ and index $\Lambda^{g^d,+}$ with the same character on $\Gamma_0(l_{g^d}^+)$. To summarize:

\begin{lemma}\label{lem:modularity}
Let $d|\hat{n}_g$. If $\xi_d=1$ on $\Lambda^{g^d}$, then $\phi_{g,d}$ is a weakly holomorphic Jacobi form of weight $0$ and index $\Lambda^{g}$ with trivial character on $\Gamma_0(N_{g^d})$, where $N_{g^d}$ is the level of $g^d$. If $\xi_d\neq 1$ on $\Lambda^{g^d}$, then $\phi_{g,d}$ is a weakly holomorphic Jacobi form of weight $0$ and index $\Lambda^{g}$ with trivial character on $\Gamma_0(N_{g^d}')$, where $N_{g^d}'$ is the least common multiple of $N_{g^d}$ and $l_{g^d}^+$.
\end{lemma}
\begin{proof}
The proof follows from Proposition \ref{prop:Jacobi-input}, Lemma \ref{lem:Jacobi-untwisted} and the discussions above. 
\end{proof}

\begin{lemma}
Let $g$ be an element of $\mathrm{Co}_0$ without a nice lift and let $\hat{n}_g$ be the order of its standard lift. For each $d|n_g$ for which $\xi_d\neq 1$ on $\Lambda^{g^{d}}$, the level of $\Lambda^{g^d, +}$ is $$l_{g^d}^+ = \frac{\hat{n}_g}{d}.$$
\end{lemma}
\begin{proof}
This can be checked by cases using the list of cycle shapes and divisors $d$ in Lemma \ref{lem:not-nice}.
\end{proof}

\begin{proof}[Proof of Theorem \ref{th:moonshine}]
(1) Note that $\phi_{g,d}=\phi_{g,(d,\hat{n}_g)}$ for any $d|N_g$. By Lemma \ref{lem:modularity}, to complete the proof it suffices to verify that for any $d|N_g$ the number $N_{g^{(d, \hat{n}_g)}}$, or $N_{g^{(d,\hat{n}_g)}}'$ respectively, divides $N_g/d$. This is straightforward to check by cases. 

\vspace{2mm}

(2) Since every $\phi_{g,d}$ has integral Fourier expansion at infinity, we conclude from Lemma \ref{lem:rationality-2} that the inverse image $F_g$ of $(\phi_{g,d})_{d|N_g}$ under the isomorphism $\mathbb{J}$ (see Theorem \ref{th:iso}) has rational principal part as a vector-valued modular form on $U(N_g)\oplus \Lambda^g$. Therefore, there exists a positive integer $N$ such that $N F_g$ has integral principal part. It follows that $\Borch(F_g)$ is well-defined as a multi-valued $N$th root of a meromorphic modular form on $U(N_g)\oplus U\oplus \Lambda^g$. 

\vspace{2mm}

For any $d | N_g$, the Fourier expansion of $\phi_{g, d}$ begins $$\phi_{g, d}(\tau, \mathfrak{z}) = q^{-1} + \mathrm{tr}(g^d) + O(q),$$ (note that when $\xi_d \neq 1$ the  part of $\phi_{g, d}$ coming from $\Lambda^{g^d, -}$ contributes nothing to the $q^{-1}$ and $q^0$ terms), and in particular $$\phi_{g, N_g}(\tau, \mathfrak{z}) = q^{-1} + 24 + O(q).$$ Theorem \ref{th:FJ-level-N} implies that the Weyl vector of $\Borch(F_g)$ is $(1, 0, 0)$ and that the weight of $\Borch(F_g)$ is 
\begin{align*}
k &= \frac{1}{2}\sum_{d|N} \mathrm{mult}_d(0,0)\\
&= \frac{1}{2}\sum_{d|N_g} \sum_{t|d} \frac{\mu(d/t)}{d} \mathrm{tr}(g^t) \\
&= \frac{1}{2}\sum_{d|N_g} b_d = \frac{1}{2} \rank(\Lambda^g).
\end{align*}
(Here we have used the formula \eqref{eq:b_m}.) The multiplicity of $\alpha=(n,\ell,m)\in L_g'$ (recall $L_g=U\oplus \Lambda^g$) as an exponent in the product expansion of $\Borch(F_g)$ (see \eqref{eq:B1}) is given by
\begin{align*}
\mathrm{mult}(n,\ell,m)&= \sum_{bd|(n,\ell,m,N_g)} \frac{\mu(b)}{bd} \mathrm{Tr}\Big((\hat{g})^{(d,\hat{n}_g)}\big|\widetilde{E}_{(nm/(b^2d^2),\ell / bd, 1)}\Big) \\
&= \sum_{bd|(n,\ell,m,N_g)} \frac{\mu(b)}{bd} \mathrm{Tr}\Big((\hat{g})^{(d,\hat{n}_g)}\big|\widetilde{E}_{(n/bd, \ell / bd, m/bd)}\Big)\\
&= \sum_{bd|(n,\ell,m,N_g)} \frac{\mu(b)}{bd} \mathrm{Tr}\Big((\hat{g})^{(d,\hat{n}_g)}\big|\widetilde{E}_{\alpha/bd}\Big).
\end{align*}
By Lemma \ref{lem:value-E}, if $\alpha \not\in L_g$, i.e. $\ell \not\in \Lambda^g$, then $\mathrm{Tr}\big((\hat{g})^d\big|\widetilde{E}_{\alpha/bd}\big)=0$, and therefore $\mathrm{mult}(n,\ell,m)=0$. Using Remark \ref{rk:match} we find 
\begin{equation}\label{eq:mult=mult}
\sum_{bd|(\alpha,N_g)} \frac{\mu(b)}{bd} \mathrm{Tr}\Big((\hat{g})^{(d,\hat{n}_g)}\big|\widetilde{E}_{\alpha/bd}\Big) = \sum_{bd|(\alpha,\hat{n}_g)} \frac{\mu(b)}{bd} \mathrm{Tr}\Big((\hat{g})^{d}\big|\widetilde{E}_{\alpha/bd}\Big).
\end{equation}
In other words, the product expansion of the Borcherds lift  $\Borch(F_g)$ is exactly the twisted denominator function \eqref{eq:twisted-denom}. It remains to prove that $\Borch(F_g)$ is single-valued and has no poles. This is obvious when $\Lambda^g=\{0\}$, since the Weyl group is generated by the exchange of $\tau$ and $\omega$ so the additive side of the twisted denominator identity yields 
$$
\Borch(F_g)(\omega,\tau) = \eta_{g}(\tau) - \eta_g(\omega). 
$$

We treat the case $\Lambda^g\neq \{0\}$ by an argument similar to the proof of \cite[Theorem 2.1]{WW23}. We use notations in \cite[Section 2]{WW23} and set $M_g:=U(N_g)\oplus U\oplus \Lambda^g$. Recall that $N$ is a positive integer such that $NF_g$ has integral principal part. Let $v^\perp$ be a divisor of $\Borch(NF_g)=\Borch(F_g)^N$ of multiplicity $d_v\neq 0$, where $v\in M_g\otimes\RR$ with $(v,v)=1$. Fix a $\mathbb{R}$-basis $e_1,...,e_l$ of the lattice $U\oplus \Lambda^g$ for which 
$$
(e_1,e_1)=-1, \quad (e_2, e_2) = ... = (e_{l}, e_{l}) = 1, \quad l=\rank(\Lambda^g)+2.
$$ 
There exists $\sigma\in \Orth^+(M_g\otimes\RR)$ such that $\sigma(v)=e_l$. Using $\sigma$ we can map the quadratic divisor $v^{\perp}$ biholomorphically to $e_l^{\perp}$ and write 
$$
\left(\Borch(F_g)\big|_{\frac{1}{2}\rank(\Lambda^g)}\sigma\right) (z)=f(z_1,...,z_{l-1})z_l^{d_v/N} + O\big(z_l^{d_v/N +1}\big), 
$$
where  $z_i=(z,e_i)$ and $f(z_1,...,z_{l-1})\neq 0$.
Let $\mathbf{\Delta}$ denote the Laplace operator on the tube domain determined by $U(N_g)$. The additive side of the twisted denominator identity shows that $\Borch(F_g)$ is singular at the $0$-dimensional cusp related to $U(N_g)$, i.e. $\mathbf{\Delta}(\Borch(F_g))=0$. We then deduce that
$$
\mathbf{\Delta}\Big(\Borch(F_g)\big|\sigma\Big) = \Big(\mathbf{\Delta}\big(\Borch(F_g)\big)\Big)\big|\sigma = 0.
$$ 
It follows that $d_v$ is a multiple of $N$ and further $d_v/N=1$, otherwise
\begin{align*} 
\mathbf{\Delta}\Big(\Borch(F_g)\big|\sigma\Big)(z)  &= \frac{1}{2}\left( -\frac{\partial^2}{\partial z_1^2} + \sum_{i=2}^l \frac{\partial^2}{\partial z_i^2} \right) \Big[ f(z_1,...,z_{l-1}) z_l^{d_v/N} + O\big(z_l^{d_v/N+1}\big) \Big] \\ &= \frac{d_v}{2N}\Big( \frac{d_v}{N} -1 \Big) f(z_1,...,z_{l-1}) z_l^{d_v/N-2} + O\big(z_l^{d_v/N-1}\big)
\end{align*} 
would be nonzero, a contradiction. 
Therefore, $\Borch(F_g)$ is single-valued and has no poles.
\end{proof}

\begin{remark}
The vector-valued modular form corresponding to $(\phi_{g,d})_{d|N_g}$ has zero component
\begin{align*}
c_{(0,0),0}(\tau) &= \frac{1}{N} \sum_{d \m N} f_{(d,0),0}(\tau) = \frac{1}{N} \sum_{d \m N} \sum_{n\in \ZZ} \phi_{g,d}(n,0) q^n \\
&= \frac{1}{N} \sum_{d \m N}  \frac{\theta_{\xi_d,\gamma_d(0)}(\tau)}{\eta_{g^d}(\tau)} 
 = \frac{1}{N} \sum_{d \m N}  \frac{\theta_{\Lambda^{g,d}}(\tau)}{\eta_{g^d}(\tau)} \\
& = \frac{1}{N}\sum_{d|N}\varphi(N/d) \frac{\theta_{\Lambda^{g,d}}(\tau)}{\eta_{g^d}(\tau)} =
\sum_{k|N} \sum_{d|k} \frac{\mu(k/d)}{k} \frac{\theta_{\Lambda^{g,d}}(\tau)}{\eta_{g^d}(\tau)}, 
\end{align*}
where $\theta_{L}(\tau) = \sum_{v \in L} q^{v^2/2} $ is the theta function of the lattice $L$. This proves a conjecture of Scheithauer \cite[\S 8]{Sch04} on the zero component of the input.
\end{remark}

\begin{remark}
By Theorem \ref{th:FJ-level-N}, for any $g\in \mathrm{Co}_0$ the twisted denominator viewed as a Borcherds product on $U(N_g)\oplus U\oplus \Lambda^g$ has Fourier--Jacobi expansion
\begin{align*}
    \Borch(F_g)(Z)&=\eta_g(\tau)\cdot \exp\left( -\sum_{d|N_g} \Grit\big(d^{-1}\phi_{g,d}\big)(dZ)  \right) \\
    &= \eta_g(\tau) - \Theta_g(\tau,\mathfrak{z})\cdot s + O(s^2), \quad s = e^{2\pi i w}.
\end{align*}
\end{remark}

\begin{remark}\label{rk:VOA}
In the introduction, we mentioned that the input $(\phi_{g,d})_{d|N_g}$ or $F_g$ is the vector-valued character of a vertex algebra. We describe this more precisely. Let $\Lambda_g$ be the orthogonal complement of $\Lambda^g$ in $\Lambda$ and denote by $V_{\Lambda_g}^{\hat{g}}$ the orbifold vertex operator algebra of $\Lambda_g$ associated with the restriction of $g$ to $\Lambda_g$. We write 
\begin{align*}
F_g(\tau) =& \sum_{a,b\in \ZZ / N_g\ZZ} \sum_{\gamma \in (\Lambda^g)'/\Lambda^g} c_{(a,b),\gamma}(\tau) \mathfrak{e}_{(a,b)}\otimes \mathfrak{e}_\gamma\\
=& \sum_{a,b\in \ZZ / N_g\ZZ} \sum_{\gamma \in (\Lambda^g)'/\Lambda^g} f_{(a,b),\gamma}(\tau) \mathfrak{f}_{(a,b)}\otimes \mathfrak{e}_\gamma.
\end{align*}
Let $\pi_g$ denote the natural isomorphism between the discriminant groups of $\Lambda^g$ and $\Lambda_g$. The components $f_{(a,b),\gamma}(\tau)$ and $c_{(a,b),\gamma}(\tau)$ appear to be respectively related to the twisted trace function and the character of $V_{\Lambda_g}^{\hat{g}}$ associated to $\pi_g(\gamma)$ and $(a,b)$ (see \cite[Propositions 3.5, 3.6]{Mol21}). The algebra $V_{\Lambda_g}^{\hat{g}}$ is one piece of the input of the BRST construction. This supports the conjecture that the generalized Kac--Moody superalgebras obtained by twisting the fake monster algebra have natural constructions as the BRST cohomology of suitable vertex algebras. (See also Problem 3 of \cite{Bor92}.)
\end{remark}

\begin{remark}
As a Borcherds product on $U(N_g)\oplus U \oplus \Lambda^g$, the twisted denominator function $\Phi_g$ is reflective if $n_g = N_g$. There are $55$ conjugacy classes $[g]$ in $\mathrm{Co}_0$, with $54$ distinct cycle shapes, for which the level equals the order and for which $\Lambda^g \ne \{0\}$. The principal parts of the input into the Borcherds lift are attached as ancillary material and are easily checked to be reflective by cases. 

\vspace{2mm}

If $n_g\neq N_g$ then $\Phi_g$ does not seem to be reflective on $U(N_g)\oplus U\oplus \Lambda^g$. 
There are $17$ conjugacy classes $[g]$ in $\mathrm{Co}_0$, with $16$ distinct cycle shapes, for which the level is greater than the order and for which $\Lambda^g \ne \{0\}$. We have verified that the product is not reflective for the cycle shapes $2^{12}$, $3^8$, $2^4 4^4$, $2^4 4^{-4} 8^4$, $2^2 4^{-1} 8^{-1} 16^2$, $2^3 6^3$, $4^6$, $6^4$, $2^2 10^2$ and $3^1 21^1$. The principal parts of the input into the ten cycle shapes were also computed in the ancillary material, and all contain a non-reflective term of the form $q^{-1} \mathfrak{e}_{(a, b)}\otimes \mathfrak{e}_{\gamma}$ with $b\neq 0 \m N_g$.  

\vspace{2mm}

By \cite[Theorem 1.4]{WW23}, the products $\Phi_g$ may be viewed as reflective modular forms on certain lattices contained in $(U(N_g) \oplus U \oplus \Lambda^g)\otimes\QQ$. Scheithauer \cite[\S 6]{Sch15} constructed such lattices for $g$ of cycle shapes $3^8$, $2^{12}$ and $3^1 21^1$. Scheithauer informed us that he proved the moonshine conjecture for
Conway's group for all elements $g$ with $n_g = N_g$ in \cite{Sch22+} by explicitly describing the vector-valued modular form on $U(N_g) \oplus U\oplus \Lambda^g$ which lifts to the twisted denominator identity of $g$.
\end{remark}

\bigskip
\noindent
\textbf{Acknowledgements} 
H. Wang is supported by the Institute for Basic Science (IBS-R003-D1). H. Wang thanks Valery Gritsenko for stimulating conversations, and thanks Nils Scheithauer for answering questions and valuable discussions. The authors thank Nils Scheithauer for helpful comments on an earlier version of this paper.

\bibliographystyle{plainnat}
\bibliofont
\bibliography{refs}

\end{document}